\documentclass[reqno,10pt]{amsart}
\usepackage{amsmath, amssymb, amsthm,amsfonts}
\usepackage[english]{babel}
\usepackage{calligra}
\usepackage{graphicx}
\usepackage{url}
\usepackage{epstopdf}
\usepackage[a4paper,bindingoffset=0.5cm,left=2cm,right=2cm,top=2.5cm,bottom=2cm,footskip=.8cm]{geometry}

\usepackage{tikz}
\usetikzlibrary{arrows, automata,positioning,calc,shapes,decorations.pathreplacing,decorations.markings,shapes.misc,petri,topaths}
\usepackage{tkz-berge}
  \usepackage{pgfplots}
  \pgfplotsset{compat=newest}
  \usetikzlibrary{plotmarks}
  \usepackage{grffile}
\newlength\figureheight
  \newlength\figurewidth
\pgfplotsset{%
    tick label style={font=\scriptsize},
    label style={font=\footnotesize},
    legend style={font=\footnotesize},
         every axis plot/.append style={very thick}
}
\usetikzlibrary{math,shapes.geometric}

\usepackage{rotating}
\usepackage{amsbsy,enumerate}
\usepackage{graphicx}
\usepackage{ccaption}
\usepackage{comment}
\usepackage{mathrsfs}
\usepackage{diagbox}
\usepackage{mathtools}

\newcommand{\vb}{\vspace{3.2mm}}
\renewcommand{\hat}{\widehat}

\newcommand{\Var }{\mathbb{V}\textnormal{\textrm{ar}}}

\allowdisplaybreaks

\newtheorem{lemma}{Lemma}
\newtheorem{corollary}{Corollary}
\newtheorem{assumption}{Assumption}
\newtheorem{claim}{Claim}
\newtheorem{theorem}{Theorem}
\newtheorem{remark}{Remark}
\newtheorem{example}{Example}
\newtheorem{definition}{Definition}
\newtheorem{proposition}{Proposition}

\makeatletter
\renewcommand{\fnum@figure}[1]{\textbf{\figurename~\thefigure}. }
\renewcommand{\fnum@table}[1]{\textbf{\tablename~\thetable}. }
\makeatother

\begin{document}

\title[Large deviations for acyclic networks]{Large deviations for acyclic networks \\of queues with correlated Gaussian inputs}

\author{Martin Zubeldia \and Michel Mandjes}

\maketitle

\begin{abstract}
We consider an acyclic network of single-server queues with heterogeneous processing rates. It is assumed that each queue is fed by the superposition of a large number of i.i.d.\ Gaussian processes with stationary increments and positive drifts, which can be correlated across different queues. The flow of work departing from each server is split deterministically and routed to its neighbors according to a fixed routing matrix, with a fraction of it leaving the network altogether.

\noindent
We study the exponential decay rate of the probability that the steady-state queue length at any given node in the network is above any fixed threshold, also referred to as the `overflow probability'. In particular, we first leverage Schilder's sample-path large deviations theorem to obtain a general lower bound for the limit of this exponential decay rate, as the number of Gaussian processes goes to infinity. Then, we show that this lower bound is tight under additional technical conditions. Finally, we show that if the input processes to the different queues are non-negatively correlated, non short-range dependent fractional Brownian motions, and if the processing rates are large enough, then the asymptotic exponential decay rates of the queues coincide with the ones of isolated queues with appropriate Gaussian inputs.

\vb


\vb

\noindent
{\sc Affiliations.} The authors are with the Korteweg-de Vries Institute for Mathematics, University of Amsterdam, Science Park 904, 1098 XH Amsterdam, the Netherlands; their research is partly funded by the NWO Gravitation project N{\sc etworks}, grant number 024.002.003.
MZ ({\it email}: {\tt \scriptsize m.zubeldia.suarez@tue.nl}) is also with the Department of Mathematics and Computer Science, Eindhoven University of Technology, Eindhoven, the Netherlands.
MM ({\it email}: {\tt \scriptsize m.r.h.mandjes@uva.nl}) is also with E{\sc urandom}, Eindhoven University of Technology, Eindhoven, the Netherlands, and the Amsterdam Business School, Faculty of Economics and Business, University of Amsterdam, Amsterdam, the Netherlands.

\vb

\noindent
{\sc Acknowledgments.}
Helpful comments by Sem Borst are greatly appreciated. Date: {\it \today.}

\end{abstract}

\newpage

\setcounter{tocdepth}{2}
\tableofcontents

\newpage

\section{Introduction}
Modern communication networks are complex, and handle huge amounts of data. This is especially true closer to the backbone of the networks, where large numbers of connections share the same resources. The design and operation of these networks greatly benefits from tractable theoretical models that are able to describe and predict the performance of the system. In order to obtain such tractable models, a common practice is to represent the network's nodes as single server queues with an appropriate service discipline. Moreover, given the high level of traffic aggregation, it is appealing to approximate the incoming traffic to the network by Gaussian processes \cite{KN, MvdM}. Since these networks are often operated in a regime where the packet loss probabilities are very small, there is a need for understanding the large-deviations behavior of these networks.

While a queueing network with Gaussian inputs is a rather streamlined model, the analysis of its large-deviations behavior is notoriously difficult outside the case of an isolated queue, which has been thoroughly studied \cite{BotvichDuffield,CourcoubetisWeber,singleQueue,DebickiMandjes}. The main reason for this is that, after the (initially Gaussian) incoming traffic goes through the first queue, it is no longer Gaussian. Then, when it is fed to a different queue, the analysis of this queue is significantly harder. For the special case of two queues in tandem, with work arriving only to the first queue and all the departing work of the first queue going into the second one, a useful trick involving subtracting the first queue (which has Gaussian input) from the sum of both queues (which behaves exactly as a single-server queue with a Gaussian input) yields a tractable analysis of the second queue in the tandem \cite{MichelTandemPaper}, even if it does not have a Gaussian input; see also the more refined approach in \cite{MMN} based on the delicate busy-period analysis developed in \cite{MMNvU}. However, this trick does not work for more complex networks (not even for two queues in tandem with inputs to \emph{both} queues, or when not all departures from the first queue join the second one \cite{tandemKavita}). Another factor that further complicates the analysis of complex networks is the fact that the input processes to the different queues can be correlated. This becomes a problem when the output of queues with correlated inputs are merged into another queue.

In this paper we consider acyclic networks of single-server queues, where work arrives to the queues as (possibly correlated) Gaussian processes, and where the work departing from each queue is deterministically split among its neighbors, with a fraction of it leaving the system altogether. This deterministic split of the departing work was also considered in e.g.\ \cite{GG1Networks}, and it is particularly suitable for modeling single-class networks (where all work is essentially exchangeable), or for modelling networks where all work needs to be routed to the same node (and thus where the splitting of departure streams is only performed to load balance the network).

In terms of our approach, this paper fits in the framework of the analysis of a single Gaussian queue \cite{singleQueue}, and the subsequent analysis of tandem, priority, and generalized processor sharing queues \cite{MichelTandemPaper,GPSQueues}; we refer to \cite{MichelBook} for a textbook account on Gaussian queues. In terms of our scope, this paper is perhaps most similar to \cite{GG1Networks}, where the authors obtained large-deviations results for acyclic networks of G/G/1 queues. However, in that paper there were certain limitations regarding the correlation structure of the input processes (in that they have to be independent across different queues), and regarding the structure of the network (in that any two directed paths cannot meet in more than one node).

\subsection{Our contribution}
In this paper we generalize the analysis of a pair of queues in tandem, fed by a single Gaussian process \cite{MichelTandemPaper}, to acyclic networks of single server queues, fed by (possibly correlated) Gaussian processes. As in \cite{MichelTandemPaper}, we assume that the arrival processes are the superposition of $n$ i.i.d.\ (multi-dimensional) Gaussian processes, and scale the processing rates of the servers by a factor of $n$, which corresponds to the so called `many sources regime'. In this regime, for any given node $i$, we work toward characterizing the asymptotic exponential decay rate of its `overflow probability', that is, the limit
\begin{equation}
 - \lim\limits_{n\to\infty} \frac{1}{n}\log \mathbb{P}\left(Q^{(n)}_i > nb\right) ,
\end{equation}
where $Q^{(n)}_i$ is the steady-state queue length at the $i$-th node, and $b$ is any positive threshold. In particular:
\begin{itemize}
 \item [(i)] We obtain a general lower bound on the asymptotic exponential decay rate  by leveraging the power of a generalized version of Schilder's theorem (Theorem \ref{thm:acyclicUpperLD}).
 \item [(ii)] Under additional technical conditions, we prove the tightness of the lower bound by finding the most likely sample paths, and showing that the corresponding asymptotic exponential decay rates coincide with the lower bound (Theorems \ref{thm:acyclicLowerLD1}, \ref{thm:acyclicLowerLD2}, and \ref{thm:acyclicLowerLD3}).
 \item [(iii)] We show that, if the input processes to the different queues are non-negatively correlated, non short-range dependent fractional Brownian motions, and if the processing rates are large enough, then the asymptotic exponential decay rates of the queues coincide with the ones of isolated queues with appropriate Gaussian inputs (Theorem \ref{thm:mfBmExample}).
\end{itemize}

\subsection{Organization of the paper}
The paper is organized as follows. In Section \ref{sec:model} we introduce some notation, the network model, and a few preliminaries on large-deviations theory. In Section \ref{sec:main} we present our main results. In Section \ref{sec:examples} we introduce an interesting example where the large-deviations behavior of any queue in the network coincides with the behavior of a single-server queue with Gaussian input. Finally, we conclude in Section \ref{sec:conclusion}.

\section{Model and preliminaries}\label{sec:model}

In this section we introduce some notation, the queueing network model that we analyze, and present a few preliminaries on sample-path large deviations theory.

\subsection{Notation for underlying graph}
Given a directed graph $G=(V,E)$, and a node $i\in V$, we introduce the following notation. Let
\[ \mathcal{N}_{\rm in}(i) \triangleq \big\{j\in V : (j,i)\in E \big\} \]
be the set of all inbound neighbors of $i$. Let
\[ \mathcal{P}_m(i) \triangleq \bigcup_{l=m}^{|V|} \Big\{ r \in V^l : r_l=i, \text{ and } (r_\ell,r_{\ell+1})\in E,\,\, \forall\, \ell \leq l-1 \Big\} \]
be the set of all directed paths that contain at least $m$ nodes, and end at node $i$. In particular, note that the trivial path $(i)$ is only in $\mathcal{P}_1(i)$. For any path $r\in\mathcal{P}_2(i)$, let $r_+\in\mathcal{P}_1(i)$ be the path that results from removing the node $r_1$ from the path $r$. Finally, for any path $r\in\mathcal{P}_1(i)$, let $|r|$ be the number of nodes that it contains.

\subsection{Queueing network}
In this subsection we introduce the basic structure of our queueing network. Consider a directed acyclic graph with $k$ nodes, and a scaling parameter $n\in\mathbb{Z}_+$. Each node $i$ of the graph is equipped with a single server with rate $n\mu_i$, and a queue with infinite capacity. Work arrives to the network in a number of stochastic processes, $A^{(n)}_1(\cdot),\dots,A^{(n)}_k(\cdot)$, with stationary increments and positive rates $n\lambda_1,\dots,n\lambda_k$, respectively (more details about these processes are given in Section~\ref{sec:gaussian_sources}). In particular, $A^{(n)}_i(\cdot)$ is the stream of work that enters the network at node $i$. Work departing from node $i$ is split deterministically so that, for each edge $(i,j)$ with $i\neq j$, a fraction $p_{i,j}\in [0,1]$ is routed to node $j$. The remaining fraction of the work departing from node $i$, denoted by $p_{i,i}\in[0,1]$, leaves the network; evidently, $\sum_{i} p_{i,j} = 1$. In order to simplify notation, for any directed path $r$, let us denote
\[ \Pi_r \triangleq \prod\limits_{\ell=1}^{|r|-1}  p_{r_\ell,r_{\ell+1}}. \]
In particular, we have $\Pi_{(i)}=1$.

For $s\leq t$, we interpret \[A^{(n)}_i(s,t)\triangleq A^{(n)}_i(t)-A^{(n)}_i(s)\] as the amount of exogenous work that arrived to the $i$-th node during the time interval $(s,t]$. Let $D^{(n)}_i(s,t)$ be the amount of work that departed the $i$-th node during $(s,t]$. Then, the total amount of work arriving to the $i$-th node during $(s,t]$ is
\begin{equation}\label{eq:defI}
 I^{(n)}_i(s,t) \triangleq A^{(n)}_i(s,t) + \sum\limits_{j\in\mathcal{N}_{\rm in}(i)} p_{j,i} D^{(n)}_j(s,t),
\end{equation}
recalling that $\mathcal{N}_{\rm in}(i)$ is the set of inbound neighbors of $i$. Furthermore, for $t\in\mathbb{R}$, Reich's formula states that the amount of remaining work in the $i$-th queue at time $t$ (also called `queue length') is given by
\begin{equation}\label{eq:defQ}
  Q^{(n)}_i(t)\triangleq  \sup\limits_{s< t} \left\{ I^{(n)}_i(s,t) - n\mu_i (t-s) \right\}.
\end{equation}
Moreover, we evidently have
\begin{align}\label{eq:defD}
 D^{(n)}_i(s,t) &= Q^{(n)}_i(s) + I^{(n)}_i(s,t) - Q^{(n)}_i(t).
\end{align}

Since we are interested in the steady-state of the queue lengths, we need to ensure that the service rate of each server is strictly larger than the total arrival rate to its node. This is enforced by imposing the following assumption.

\begin{assumption}\label{ass:basic}
For each $i\in\{1,\dots,k\}$, we have
\[ \sum\limits_{r\in\mathcal{P}_1(i)} \lambda_{r_1} \Pi_r < \mu_i. \]
\end{assumption}

Note that, even under Assumption \ref{ass:basic}, the existence and uniqueness of $k$-dimensional processes $D^{(n)}(\cdot)$, $I^{(n)}(\cdot)$, and $Q^{(n)}(\cdot)$ that satisfy equations \eqref{eq:defI}, \eqref{eq:defQ}, and \eqref{eq:defD} is not immediate. This will be established in Section \ref{sec:functional}, by expressing them as functionals of the exogenous arrival processes $A^{(n)}_1(\cdot),\dots,A^{(n)}_k(\cdot)$.

\subsection{Gaussian arrival processes}\label{sec:gaussian_sources}
In this subsection, we specify the nature of the exogenous arrivals to the network. Let $\{X^{(j)}(\cdot)\}_{j\in{\mathbb Z}_+}$ be a sequence of i.i.d.\ $k$-dimensional Gaussian processes with continuous sample paths and stationary increments,
and with $X^{(j)}(0)=(0,\dots,0)$, for all $j\in{\mathbb Z}_+$. Each one of these $k$-dimensional processes is characterized by its drift vector $\lambda=(\lambda_1,\dots,\lambda_k)$, where
\[ \lambda \triangleq \mathbb{E}\left[X^{(1)}(1)\right], \]
and by its covariance matrix $\Sigma:\mathbb{R}^2\to\mathbb{R}^{k\times k}$, where
\[ \Sigma_{i,j}(t,s) = {\mathbb C}{\rm ov}\left( X_i^{(1)}(t),\, X_j^{(1)}(s) \right).  \]
Throughout this paper, we assume that the process $A^{(n)}(\cdot)\triangleq\big( A^{(n)}_1(\cdot),\dots,A^{(n)}_k(\cdot) \big)$ is a $k$-dimensional Gaussian process such that
\begin{equation}\label{eq:superposition}
  A^{(n)}_i(\cdot) = \sum\limits_{j=1}^n X_i^{(j)}(\cdot),
\end{equation}
for all $i\in\{1,\dots,k\}$. Therefore, $A^{(n)}(\cdot)$ also has continuous sample paths and stationary increments, and satisfies $A^{(n)}(0)=(0,\dots,0)$. Moreover, the $k$-variate process $A^{(n)}(\cdot)$ has drift vector $n\lambda$, and covariance matrix $n\Sigma$.

\begin{remark} {\em
Equation \eqref{eq:superposition} corresponds to the setting where the arrival processes are a superposition of individual streams, which is also called the ``many-sources regime" \cite{manySources}.
}\end{remark}

Finally, the following assumption is in place. It is required for a generalized version of Schilder's theorem to hold, which is introduced in the following subsection.

\begin{assumption}\label{ass:acyclic}
$ $
\begin{itemize}
\item [(i)] The covariance matrix $\Sigma$ is differentiable.
\item [(ii)] For every $i,j\in\{1,\dots,k\}$, we have
\[ \lim\limits_{t^2+s^2\to\infty} \frac{\Sigma_{i,j}(t,s)}{t^2+s^2} = 0.  \]
\end{itemize}
\end{assumption}

\subsection{Sample-path large deviations}
In this paper, our aim is to study the limit
\begin{equation}\label{eq:basicApproximation}
 I_i(b) \triangleq - \lim\limits_{n\to\infty} \frac{1}{n}\log \mathbb{P}\left(Q^{(n)}_i > nb\right) ,
\end{equation}
where $Q^{(n)}_i$ is the steady-state queue length of the $i$-th node, and $I_i:\mathbb{R}_+\to\mathbb{R}_+^k$ is a function that only depends on the server rates $\mu\triangleq (\mu_1,\dots,\mu_k)$, on the drift vector $\lambda$, and on the covariance matrix $\Sigma$. In order to do this, we rely on a sample-path large deviations principle for centered Gaussian processes, based on the generalized Schilder's theorem. Before stating this theorem, we introduce its framework.

First, we introduce the sample-path space
\[ \Omega^k \triangleq \left\{ \omega:\mathbb{R}\to\mathbb{R}^k, \,\,\text{continuous},\,\, \omega(0)=(0,\dots,0),\,\, \lim\limits_{t\to\infty} \frac{\|\omega(t)\|_2}{1+|t|}=\lim\limits_{t\to-\infty} \frac{\|\omega(t)\|_2}{1+|t|}=0 \right\}, \]
equipped with the norm
\[ \|\omega \|_{\Omega^k} \triangleq \sup\left\{ \frac{\|\omega(t)\|_2}{1+|t|} : t\in\mathbb{R} \right\}, \]
which is a separable Banach space \cite{MannersaloNorros2001}. Next, we introduce the Reproducing Kernel Hilbert Space ({\sc rkhs}) $\mathcal{R}^k\subset\Omega^k$ (see \cite{AdlerBook}  for more details) induced by using the covariance matrix $\Sigma(\cdot,\cdot)$ as the kernel. In order to define it, we start from the smaller space
\[ \mathcal{R}^k_* \triangleq \text{span}\left\{ \Sigma(t, \cdot).v : t\in\mathbb{R},\, v\in \mathbb{R}^k \right\}, \]
with the inner product $\langle \cdot, \cdot \rangle_{\mathcal{R}^k}$ defined as
\[ \big\langle \Sigma(t,\cdot)u,\, \Sigma(s,\cdot)v \big\rangle_{\mathcal{R}^k} \triangleq u^\top\Sigma(t,s)v, \]
for all $t,s\in \mathbb{R}$ and $u,v\in\mathbb{R}^k$. The closure of $\mathcal{R}^k_*$ with respect to the topology induced by its inner product is the {\sc rkhs} $\mathcal{R}^k$. Using this inner product and its corresponding norm $\|\cdot\|_{\mathcal{R}^k}$, we define a rate function by
\[
\mathbb{I}(\omega) \triangleq
\begin{cases}
\frac{1}{2}\|\omega\|^2_{\mathcal{R}^k}, & \text{if } \omega\in\mathcal{R}^k,\\
\infty, & \text{otherwise}.
\end{cases}
\]

\begin{remark} {\em
In \cite{MannersaloNorros2001,GPSQueues}, the authors defined an appropriate multi-dimensional {\sc rkhs} as the product of single-dimensional spaces that use the individual variance functions as kernels. There this could be done because the different coordinates of the multi-dimensional Gaussian process of interest were assumed independent. In our case, since the coordinates of our Gaussian process of interest need not be independent, we needed to define the multi-dimensional space directly, using the whole covariance matrix as the kernel. When the coordinates are indeed independent, both definitions are equivalent.
}\end{remark}

Under the framework define above, the following sample-path large deviations principle holds.

\begin{theorem}[Generalized Schilder \cite{schilder}]\label{thm:Schilder}
Under Assumption \ref{ass:acyclic}, the following hold.
\begin{itemize}
  \item [(i)] For any closed set $F\subset\Omega^k$,
  \[ \limsup\limits_{n\to\infty} \frac{1}{n}\log  \mathbb{P}\left( \frac{A^{(n)}(\cdot)-n\lambda \,\cdot\,}{n} \in F \right) \leq -\inf\limits_{\omega\in F} \big\{ \mathbb{I}(\omega) \big\}.  \]
  \item [(ii)] For any open set $G\subset\Omega^k$,
  \[ \liminf\limits_{n\to\infty} \frac{1}{n}\log  \mathbb{P}\left( \frac{A^{(n)}(\cdot) -n\lambda \,\cdot\,}{n} \in G \right)   \geq -\inf\limits_{\omega\in G} \big\{ \mathbb{I}(\omega) \big\}.  \]
\end{itemize}
\end{theorem}

Schilder's theorem typically only gives implicit results, as it is often hard to explicitly compute the infimum over the set of sample paths. However, as in \cite{MichelTandemPaper,GPSQueues,MannersaloNorros2001}, we will leverage the properties of our {\sc rkhs} to obtain explicit results.

\section{Main results}\label{sec:main}
In this section we will establish large-deviations results for the steady-state queue-length distributions. In particular, we will use Theorem \ref{thm:Schilder} to show that, for any $\{1,\dots,k\}$, and for every $b>0$, the limit
\begin{equation}\label{eq:objective}
 -\lim\limits_{n\to\infty} \frac{1}{n}\log \mathbb{P}\left(Q^{(n)}_i > nb\right)  \end{equation}
exists, and to find (tight) bounds for it. The first step is to express this probability as a function of the Gaussian arrival processes (Section \ref{sec:functional}), and to show that the limit exists (Section \ref{sec:existence}). Second, we obtain a general upper bound for this limit (Section \ref{sec:lowerBound}), and prove that it is tight under additional technical assumptions (Section~\ref{sec:upperBound}). The arguments largely follow the same structure as the arguments for the analysis of the second queue in a tandem \cite{MichelTandemPaper}, but without the simplifications that come from having only two queues in tandem, with arrivals only to the first one.

\subsection{Overflow probability as a function of the arrival processes}\label{sec:functional}
In this subsection we obtain a set $\mathcal{E}^i(b)$ of sample paths such that
\[ \mathbb{P}\left(Q^{(n)}_i > nb\right) = \mathbb{P}\left(\frac{A^{(n)}(\cdot)-n\lambda \,\cdot\,}{n} \in \mathcal{E}^i(b)\right). \]
By Reich's formula, we have
\begin{align*}
  \mathbb{P}\left(Q^{(n)}_i > nb\right) &= \mathbb{P}\left(\sup\limits_{t< 0} \left\{ I^{(n)}_i(t,0) + n\mu_i t \right\} > nb\right) = \mathbb{P}\left(\exists\, t< 0 : I^{(n)}_i(t,0) + n\mu_i t > nb\right),
\end{align*}
where $I^{(n)}_i(t,0)$ is the total amount of work that arrived to the $i$-th queue in the time interval $(t,0]$. If $i$ is a node with no inbound neighbors, i.e., if $\mathcal{N}_{\rm in}(i)=\emptyset$, we have that $I^{(n)}_i(t,0)=-A^{(n)}_i(t)$, and thus
\begin{align*}
  \mathbb{P}\left(Q^{(n)}_i > nb\right) &= \mathbb{P}\left(\exists\, t< 0 : n\mu_i t -A^{(n)}_i(t) > nb\right) \\
  &= \mathbb{P}\left( \frac{A^{(n)}(\cdot)-n\lambda \,\cdot\,}{n} \in \Big\{f\in\Omega^k : \exists\, t< 0,\, (\mu_i-\lambda_i) t - f_i(t) > b \Big\}\right).
\end{align*}
In this case, a large-deviations analysis can be performed through a straightforward application of Schilder's theorem (this is exactly the same as in the case of an isolated Gaussian queue \cite{singleQueue}). However, in general the input process is the sum of the local Gaussian arrival process, and the departure processes of its inbound neighbors, which are not Gaussian. In the following lemma we obtain the input process as a functional of the exogenous arrival processes of all the upstream nodes.

\begin{lemma}\label{lem:inputFormula}
For each $i\in\{1,\dots,k\}$, and for all $t<0$, we have
\begin{align}
 I^{(n)}_i(t,0) = A^{(n)}_i(t,0) +&  \sup\limits_{{\boldsymbol t}\in\mathcal{T}_i(t)} \left\{ \sum\limits_{r\in\mathcal{P}_2(i)}  \left[ A^{(n)}_{r_1}({\boldsymbol t}_{r},0)+ n\mu_{r_1}({\boldsymbol t}_{r}-{\boldsymbol t}_{r_+}) \right] \Pi_r \right\} \label{eq:inputProcess} \\
-& \sup\limits_{{\boldsymbol s}\in\mathcal{T}_i(0)} \left\{ \sum\limits_{r\in\mathcal{P}_2(i)} \left[ A^{(n)}_{r_1}({\boldsymbol s}_{r},0)+ n\mu_{r_1}({\boldsymbol s}_{r}-{\boldsymbol s}_{r_+}) \right] \Pi_r \right\}, \nonumber
\end{align}
where
\begin{align*}
 \mathcal{T}_i(t) &\triangleq \Big\{ {\boldsymbol t}\in\mathbb{R}^{\mathcal{P}_1(i)}: {\boldsymbol t}_i = t  \quad \text{and} \quad {\boldsymbol t}_r < {\boldsymbol t}_{r_+}, \,\, \forall\, r\in\mathcal{P}_2(i) \Big\}.
\end{align*}
\end{lemma}
The proof is given in Appendix \ref{app:proofInputFormula}, and consists of solving a recursive equation on the input processes by using induction on the maximum length of paths that end in node $i$.

\begin{remark} {\em \label{rem:busyPeriods}
Let ${\boldsymbol t}^*$ and ${\boldsymbol s}^*$ be finite optimizers of the two suprema in  \eqref{eq:inputProcess} over the closure of their domains. These have the following interpretation: for each path $r\in\mathcal{P}_2(i)$, the time ${\boldsymbol t}^*_{r}$ (respectively, ${\boldsymbol s}^*_{r}$) is the starting point of the busy period of the $r_1$-th queue that contains the time ${\boldsymbol t}^*_{r_+}$ (respectively, ${\boldsymbol s}^*_{r_+}$). Then, since ${\boldsymbol t}_i=t <0$ and ${\boldsymbol s}_i=0$, it follows that ${\boldsymbol t}^*_{r}\leq {\boldsymbol s}^*_{r}$, for all $r\in\mathcal{P}_1(i)$. Combining this with  \eqref{eq:inputProcess}, and using the continuity of $A^{(n)}(\cdot)$, we obtain
\begin{align*}
  I^{(n)}_i(t,0) = A^{(n)}_i(t,0) &- n\left( \sum\limits_{j\in\mathcal{N}_{\rm in}(i)} \mu_j p_{j,i} \right) t \\
  &+ \sup\limits_{{\boldsymbol t}\in\mathcal{T}_i}\left\{ \sum\limits_{r\in\mathcal{P}_2(i)}  \left[ A^{(n)}_{r_1}({\boldsymbol t}_{r},0)+ n \left(\mu_{r_1} - \sum\limits_{j\in\mathcal{N}_{\rm in}(r_1)}  \mu_j p_{j,r_1} \right){\boldsymbol t}_{r} \right] \Pi_r  \right. \\
 &  \left.- \sup\limits_{{\boldsymbol s}\in\mathcal{S}_i({\boldsymbol t})} \left\{ \sum\limits_{r\in\mathcal{P}_2(i)} \left[ A^{(n)}_{r_1}({\boldsymbol s}_{r},0)+ n\left(\mu_{r_1} - \sum\limits_{j\in\mathcal{N}_{\rm in}(r_1)}  \mu_j p_{j,r_1} \right) {\boldsymbol s}_{r} \right] \Pi_r \right\} \right\},
\end{align*}
where
\begin{align*}
 \mathcal{T}_i &\triangleq \Big\{ {\boldsymbol t}\in\mathbb{R}^{\mathcal{P}_1(i)}: {\boldsymbol t}_i < 0  \quad \text{and} \quad {\boldsymbol t}_r < {\boldsymbol t}_{r_+}, \,\, \forall\, r\in\mathcal{P}_2(i) \Big\},\\
 \mathcal{S}_i({\boldsymbol t}) &\triangleq \Big\{ {\boldsymbol s}\in\mathbb{R}^{\mathcal{P}_1(i)}: {\boldsymbol s}_i=0  \quad \text{and} \quad {\boldsymbol t}_{r} < {\boldsymbol s}_r < {\boldsymbol s}_{r_+}, \,\, \forall\, r\in\mathcal{P}_2(i) \Big\}.
\end{align*}
Note that the continuity of $A^{(n)}(\cdot)$ is what allows us to have the condition ${\boldsymbol t}_{r} < {\boldsymbol s}_{r}$ instead of ${\boldsymbol t}_{r} \leq {\boldsymbol s}_{r}$. This distinction will be convenient later.
}\end{remark}

We now state the main result of this subsection.

\begin{theorem}\label{thm:probAsFunctOfInputs}
  For each $i\in\{1,\dots,k\}$, and for every $b>0$, we have
  \begin{align}\label{repr_p}
 \mathbb{P}\left( Q_i^{(n)} > b n \right) &= \mathbb{P}\left( \frac{A^{(n)}(\cdot)-n\lambda\, \cdot\,}{n} \in \mathcal{E}_i(b) \right),
\end{align}
where
\begin{align*}
 &\mathcal{E}_i(b) \triangleq \left\{ f\in\Omega^k: \exists\, {\boldsymbol t}\in\mathcal{T}_i : \forall\, {\boldsymbol s}\in\mathcal{S}_i({\boldsymbol t}),\,\, f_i({\boldsymbol t}_i) + \sum\limits_{r\in\mathcal{P}_2(i)} \Big[f_{r_1}({\boldsymbol t}_{r}) - f_{r_1}({\boldsymbol s}_{r})\Big] \Pi_r \right. \\
 &\qquad\qquad\qquad\qquad\qquad\qquad\qquad \left. > b - \sum\limits_{r\in\mathcal{P}_1(i)} \left[ \left(\mu_{r_1}-\lambda_{r_1} - \sum\limits_{j\in\mathcal{N}_{\rm in}(r_1)} \mu_j p_{j,r_1}\right) \big({\boldsymbol t}_{r}-{\boldsymbol s}_{r}\big)\right] \Pi_r \right\}.
\end{align*}
\end{theorem}
The proof follows immediately from Reich's formula and Lemma \ref{lem:inputFormula}, and it is given in Appendix \ref{app:probAsFunctOfInputs}.

\subsection{Decay rate of the overflow probability}\label{sec:existence}
In this subsection we establish the existence of the limit
\begin{equation*}
 -\lim\limits_{n\to\infty} \frac{1}{n}\log \mathbb{P}\left( Q_i^{(n)} > b n \right) ,% = \inf\limits_{f\in \mathcal{S}} \big\{ \mathbb{I}(f) \big\}.
\end{equation*}
for all $b>0$. Recall that Theorem \ref{thm:probAsFunctOfInputs} states that $\mathbb{P}( Q_i^{(n)} > b n)$ satisfies \eqref{repr_p}, where $\mathcal{E}^i(b)$ is an open set of the path space $\Omega^k$. Then, by Schilder's theorem (Theorem \ref{thm:Schilder}), we have
\begin{align*}
  -\liminf\limits_{n\to\infty} \frac{1}{n}\log\left( \mathbb{P}\left( \frac{A^{(n)}(\cdot)-n\lambda \,\cdot\,}{n} \in \mathcal{E}^i(b) \right) \right) &\leq \inf\limits_{f\in \mathcal{E}^i(b)} \big\{ \mathbb{I}(f) \big\},
\end{align*}
and
\begin{align*}
  -\limsup\limits_{n\to\infty} \frac{1}{n}\log\left( \mathbb{P}\left( \frac{A^{(n)}(\cdot)-n\lambda \,\cdot\,}{n} \in \overline{\mathcal{E}^i(b)} \right) \right) &\geq \inf\limits_{f\in \overline{\mathcal{E}^i(b)}} \big\{ \mathbb{I}(f) \big\}.
\end{align*}
Then, the existence of the limit is equivalent to showing that $\mathcal{E}^i(b)$ is an $\mathbb{I}$-continuity set, which is stated in the following proposition. The proof follows the lines of the proof of  \cite[Thm.\ 3.1]{MichelTandemPaper}, and it is thus omitted.

\begin{proposition}\label{thm:Icontinuity}
For each $i\in\{1,\dots,k\}$, and for every $b>0$, we have
\begin{equation}\label{eq:Icontinuity}
 -\lim\limits_{n\to\infty} \frac{1}{n}\log\left( \mathbb{P}\left( Q_i^{(n)} > b n \right) \right) = \inf\limits_{f\in \overline{\mathcal{E}_i(b)}} \big\{ \mathbb{I}(f) \big\} = \inf\limits_{f\in \mathcal{E}_i(b)} \big\{ \mathbb{I}(f) \big\}.
\end{equation}
\end{proposition}

Since the existence of the decay rate of interest given in  \eqref{eq:basicApproximation} has been established now, in the following subsections we focus on finding lower and upper bounds on it.

\subsection{Lower bound on the decay rate}\label{sec:lowerBound}
In this subsection we present a general lower bound for the asymptotic exponential decay rate of the overflow probability in steady state. We start by introducing some notation. Given a vector $v$ and a scalar $a$, we denote $v-(a,\dots,a)$ as $v-a$. For each node $i\in\{1,\dots,k\}$, we denote
\begin{align*}
\hat{A}_i(t) &\triangleq \frac{A_i^{(n)}(t) - n\lambda_i t}{\sqrt{n}}.
\end{align*}
Note that $\hat A(\cdot)$ is a $k$-dimensional Gaussian process with zero mean, and covariance matrix $\Sigma$. For each node $i\in\{1,\dots,k\}$,
\begin{align*} \overline \lambda_i &\triangleq \sum\limits_{r\in\mathcal{P}_1(i)} \lambda_{r_1} \Pi_r, \\
\bar A_i({\boldsymbol s},{\boldsymbol t}) &\triangleq \sum\limits_{r\in\mathcal{P}_1(i)} \Big[ \hat A_{r_1}({\boldsymbol t}_{r}) - \hat A_{r_1}({\boldsymbol s}_{r}) \Big] \Pi_r. \end{align*}
Moreover, let us define the functions
\begin{align*}
 k_b^i({\boldsymbol t},{\boldsymbol s}) &\triangleq \mathbb{E}\left[\left. \bar A_i({\boldsymbol t}-{\boldsymbol t}_i,{\boldsymbol s}) \,\right|\, \bar A_i({\boldsymbol t}-{\boldsymbol t}_i,{\boldsymbol t}) =b - \left(\mu_i-\overline\lambda_i\right) {\boldsymbol t}_i \right],\\
  h_b^i({\boldsymbol t},{\boldsymbol s}) &\triangleq \mathbb{E}\left[ \bar A_i({\boldsymbol t}-{\boldsymbol t}_i,{\boldsymbol s}) \,\left|\, \bar A_i({\boldsymbol s},{\boldsymbol t}) = b - \left(\mu_i-\overline\lambda_i\right) {\boldsymbol t}_i - c_i({\boldsymbol t},{\boldsymbol s}) \right.\right],
\end{align*}
where
\begin{align*}
  c_i({\boldsymbol t},{\boldsymbol s}) & \triangleq \left( \overline{\lambda}_i - \lambda_i - \sum\limits_{j\in\mathcal{N}_{\rm in}(i)} \mu_j p_{j,i} \right){\boldsymbol t}_i + \sum\limits_{r\in\mathcal{P}_2(i)} \left(\mu_{r_1}-\lambda_{r_1} - \sum\limits_{j\in\mathcal{N}_{\rm in}(r_1)} \mu_j p_{j,r_1} \right)\big({\boldsymbol t}_{r}-{\boldsymbol s}_{r}\big) \Pi_r.
\end{align*}
Note that $c_i({\boldsymbol t},{\boldsymbol t}-{\boldsymbol t}_i)=0$.

Using the above notation, we now state our lower bound.

\begin{theorem}\label{thm:acyclicUpperLD}
  Under Assumptions \ref{ass:basic} and \ref{ass:acyclic}, for each $i\in\{1,\dots,k\}$ and for every $b > 0$,
\begin{equation*}
   -\lim\limits_{n\to\infty} \frac{1}{n}\log  \mathbb{P}\left( Q_i^{(n)} > b n  \right) \geq \inf\limits_{{\boldsymbol t}\in\mathcal{T}_i} \sup\limits_{{\boldsymbol s}\in\mathcal{S}_i({\boldsymbol t})} \Big\{ I^i_b({\boldsymbol t},{\boldsymbol s}) \Big\},
  \end{equation*}
where
\begin{equation}\label{eq:acyclicExponent}
I^i_b({\boldsymbol t},{\boldsymbol s}) \triangleq
\begin{cases}
\frac{\Big[ b - \left(\mu_i-\overline\lambda_i\right) {\boldsymbol t}_i \Big]^2}{2\,\Var \Big( \bar A_i({\boldsymbol t}-{\boldsymbol t}_i,{\boldsymbol t}) \Big)}, & \begin{array}{l}
\text{if } k_b^i({\boldsymbol t},{\boldsymbol s}) < c_i({\boldsymbol t},{\boldsymbol s}), \\
\begin{array}{l}\text{or} \end{array}\quad {\boldsymbol s}={\boldsymbol t}-{\boldsymbol t}_i,\end{array} \\
\frac{\Big[b - \left(\mu_i-\overline\lambda_i\right) {\boldsymbol t}_i - c_i({\boldsymbol t},{\boldsymbol s}) \Big]^2}{2\,\Var \Big( \bar A_i({\boldsymbol s},{\boldsymbol t}) \Big)}, & \text{if } h_b^i({\boldsymbol t},{\boldsymbol s})> c_i({\boldsymbol t},{\boldsymbol s}), \\
\frac{\Big[ b - \left(\mu_i-\overline\lambda_i\right) {\boldsymbol t}_i \Big]^2}{2\,\Var \Big( \bar A_i({\boldsymbol t}-{\boldsymbol t}_i,{\boldsymbol t}) \Big)} + \frac{\Big[ k_b^i({\boldsymbol t},{\boldsymbol s})- c_i({\boldsymbol t},{\boldsymbol s}) \Big]^2}{2\,\Var \Big( \bar A_i({\boldsymbol s},{\boldsymbol t}) \,\Big|\, \bar A_i({\boldsymbol t}-{\boldsymbol t}_i,{\boldsymbol t}) = b - \left(\mu_i-\overline\lambda_i\right) {\boldsymbol t}_i \Big)}, & \begin{array}{l}\text{otherwise.}\end{array}
\end{cases}
\end{equation}
\end{theorem}
The proof is given in Appendix \ref{app:acyclicUpperLD}, and it essentially consists of two steps. First, we decompose the event $\mathcal{E}_i(b)$ given in Theorem \ref{thm:probAsFunctOfInputs} as a union of intersections of simpler events that only involve the sample paths at fixed times, and we upper bound the probability of the intersection by the probability of the least likely one. Then, we use Cram\'er's theorem to obtain the decay rate of the least likely of these simpler events by solving the additional quadratic optimization problem that arises by its application.

\begin{remark} {\em
  As part of the proof of Theorem \ref{thm:acyclicUpperLD}, it is established that conditions $k_b^i({\boldsymbol t},{\boldsymbol s}) < c_i({\boldsymbol t},{\boldsymbol s})$ or ${\boldsymbol s}={\boldsymbol t}-{\boldsymbol t}_i$, and $h_b^i({\boldsymbol t},{\boldsymbol s})>c_i({\boldsymbol t},{\boldsymbol s})$ cannot be satisfied at the same time. As a result, the three cases in the definition of $I^i_b({\boldsymbol t},{\boldsymbol s})$ are disjoint.
}\end{remark}

\begin{remark} {\em
The lower bound in Theorem \ref{thm:acyclicUpperLD} generalizes the lower bound given in \cite[Corollary 3.5]{MichelTandemPaper}, not only by generalizing the network structure from a set of tandem queues to any acyclic network of queues, but also by removing a concavity assumption on the square root of the variance of the input processes. However, the removal of this assumption makes the expression of the lower bound more convoluted, even if we restrict it to the case of a pair of queues in tandem.
}\end{remark}

\begin{remark} {\em
It is worth highlighting that, even if the bound of Theorem \ref{thm:acyclicUpperLD} is not tight, it provides an upper bound for the asymptotic exponential decay rate of overflow probability that can be used as a performance guarantee in applications.
}\end{remark}

\subsection{Tightness of the lower bound}\label{sec:upperBound}
In this subsection we obtain conditions under which the lower bound in Theorem \ref{thm:acyclicUpperLD} is tight. We present three results, one for each of the cases in the definition of $I^i_b({\boldsymbol t},{\boldsymbol s})$ in   \eqref{eq:acyclicExponent}, with different technical conditions for each case.

Let $({\boldsymbol t}^*,{\boldsymbol s}^*)$ be an optimizer of   \eqref{eq:acyclicExponent} over the closure of its domain. We first establish that, if the optimum of  \eqref{eq:acyclicExponent} is achieved in the first case, then the lower bound of Theorem \ref{thm:acyclicUpperLD} is tight under an additional technical condition. This is formalized in the following theorem.

\begin{theorem}\label{thm:acyclicLowerLD1}
  Under Assumptions \ref{ass:basic} and \ref{ass:acyclic}, the following holds. If
\begin{equation}\label{eq:case1allS}
 k_b^i\left({\boldsymbol t}^*,{\boldsymbol s}\right) < c_i\left({\boldsymbol t}^*,{\boldsymbol s}\right),
\end{equation}
for all ${\boldsymbol s}\in\mathcal{S}_i({\boldsymbol t}^*)$ such that ${\boldsymbol s} \neq {\boldsymbol t}^*- {\boldsymbol t}^*_i$, then
\begin{equation*}
   \lim\limits_{n\to\infty} \frac{1}{n}\log  \mathbb{P}\left( Q_i^{(n)}  > b n \right) = -\inf\limits_{{\boldsymbol t}\in\mathcal{T}_i} \left\{ \frac{\Big[ b - \left(\mu_i-\overline\lambda_i\right) {\boldsymbol t}_i \Big]^2}{2\,\Var \left( \bar A_i({\boldsymbol t}-{\boldsymbol t}_i,{\boldsymbol t}) \right)} \right\}.
\end{equation*}
\end{theorem}
The proof is given in Appendix \ref{app:acyclicLowerLD1}, and it essentially consists of two steps. First, we identify a most likely sample path in the {\it least} likely event of the intersection given in the decomposition of the event $\mathcal{E}_i(b)$ that was used in the proof of Theorem \ref{thm:acyclicUpperLD}. Then, we show that under the assumptions imposed this most likely sample path is in {\it all} the sets featuring in the intersection, thus implying optimality in
$\mathcal{E}_i(b)$.

Since the condition in  \eqref{eq:case1allS} requires an optimizer $t^*$ of   \eqref{eq:acyclicExponent}, it is generally hard to verify. In the following lemma we present a sufficient condition that is easier to verify.

\begin{lemma}\label{lem:sufficientCondition}
A sufficient condition for   \eqref{eq:case1allS} to hold is that
\begin{equation}\label{eq:case1simplified}
 k_b^i\left(\tilde {\boldsymbol t},{\boldsymbol s}\right) < c_i\left(\tilde {\boldsymbol t},{\boldsymbol s}\right),
\end{equation}
for all ${\boldsymbol s}\in \mathcal{S}_i\left(\tilde {\boldsymbol t}\right)$ such that ${\boldsymbol s}\neq \tilde {\boldsymbol t}- \tilde {\boldsymbol t}_i$, where
\[ \tilde {\boldsymbol t} \in \underset{{\boldsymbol t}\in\overline{\mathcal{T}_i}}{\arg\min}\left\{ \frac{\Big[ b - \left(\mu_i-\overline\lambda_i\right) {\boldsymbol t}_i \Big]^2}{2\,\Var \left( \bar A_i({\boldsymbol t}-{\boldsymbol t}_i,{\boldsymbol t})  \right)} \right\}. \]
\end{lemma}
The proof is given in Appendix \ref{app:sufficientCondition}.

\begin{remark} {\em
  Although the condition of   \eqref{eq:case1simplified} looks almost the same as the original one of   \eqref{eq:case1allS}, the key simplification is that for \eqref{eq:case1simplified} we only need $\tilde {\boldsymbol t}$ instead of ${\boldsymbol t}^*$, which is an optimizer of an easier optimization problem.
}\end{remark}

%\begin{remark} {\em
%Note that, if
%\[ \Var \left( \hat A_i({\boldsymbol t})  \right)=\Var \left( \hat A_i({\boldsymbol t}_i) + \sum\limits_{r\in\mathcal{P}(i)} \Big[ \hat A_{r_1}({\boldsymbol t}_{r}) - \hat A_{r_1}({\boldsymbol t}_{r}-{\boldsymbol t}_i) \Big] \Pi_r \right) \]
%is maximized when ${\boldsymbol t}_{r} = {\boldsymbol t}_i$, for all $r\in\mathcal{P}_2(i)$ (e.g., when the network is a tree and all arrival processes are independent), then
%\[ \lim\limits_{n\to\infty} \frac{1}{n}\log\left( \mathbb{P}\left( Q_i^{(n)}  > b n \right) \right) = \inf\limits_{{\boldsymbol t}_i< 0} \left\{ \frac{\Big[ b - \left(\mu_i-\overline\lambda_i\right) {\boldsymbol t}_i \Big]^2}{2\,\Var \left( \hat A_i({\boldsymbol t}_i) + \sum\limits_{r\in\mathcal{P}(i)} \hat A_{r_1}({\boldsymbol t}_i) \Pi_r \right)} \right\}. \]
%In this case, Theorem \ref{thm:acyclicLowerLD1} implies that the large deviations exponent of the $i$-th queue is the same as the one where all external arrival processes to upstream queues are inputs to the $i$-th queue instead.
%}\end{remark}

We now present the second result of this subsection. It asserts that, if the optimum of  \eqref{eq:acyclicExponent} is achieved in the second case, then the lower bound of Theorem \ref{thm:acyclicUpperLD} is tight under an additional technical condition.

\begin{theorem}\label{thm:acyclicLowerLD2}
Under assumptions \ref{ass:basic} and \ref{ass:acyclic}, the following holds. Suppose that
\begin{align*}
 & \mathbb{E}\left[ \bar A_i({\boldsymbol s},{\boldsymbol s}^*) \,\left|\, \bar A_i({\boldsymbol s}^*,{\boldsymbol t}^*) = b -\left(\mu_i-\overline\lambda_i\right){\boldsymbol t}_i^* - c_i({\boldsymbol t}^*,{\boldsymbol s}^*) \right. \right] \geq c_i({\boldsymbol t}^*,{\boldsymbol s}^*) - c_i({\boldsymbol t}^*,{\boldsymbol s}),
\end{align*}
for all ${\boldsymbol s}\in\mathcal{S}_i({\boldsymbol t}^*)$. If
\[ h_b^i\left({\boldsymbol t}^*,{\boldsymbol s}^*\right) > c_i\left({\boldsymbol t}^*,{\boldsymbol s}^*\right) \]
then
\begin{equation*}
   \lim\limits_{n\to\infty} \frac{1}{n}\log \mathbb{P}\left( Q_i^{(n)}   > b n \right) = -\inf\limits_{{\boldsymbol t}\in\mathcal{T}_i} \sup\limits_{{\boldsymbol s}\in\mathcal{S}_i({\boldsymbol t})} \left\{ \frac{\Big[b - \left(\mu_i-\overline\lambda_i\right) {\boldsymbol t}_i - c_i({\boldsymbol t},{\boldsymbol s}) \Big]^2}{2\,\Var \left( \bar A_i({\boldsymbol s},{\boldsymbol t}) \right)} \right\}.
  \end{equation*}
\end{theorem}
The proof is analogous to the proof of Theorem \ref{thm:acyclicLowerLD1}, and it is thus omitted.

\begin{remark} {\em
  Note that the second condition in Theorem \ref{thm:acyclicLowerLD2} is satisfied  if the first one is satisfied with strict inequality for ${\boldsymbol s}={\boldsymbol t}^*-{\boldsymbol t}_i^*$.
}\end{remark}

Finally, we show that if the optimum of  \eqref{eq:acyclicExponent} is achieved in the third case, then the lower bound of Theorem \ref{thm:acyclicUpperLD} is tight under a different additional technical condition.

\begin{theorem}\label{thm:acyclicLowerLD3}
Under Assumptions \ref{ass:basic} and \ref{ass:acyclic}, the following holds. Suppose that
\begin{align}
 & \mathbb{E}\left[ \left. \bar A_i({\boldsymbol s},{\boldsymbol s}^*) \,\right|\, \bar A_i({\boldsymbol t}^*-{\boldsymbol t}^*_i,{\boldsymbol t}^*) = b - \left(\mu_i-\overline\lambda_i\right) t^*_i;\,\, \bar A_i({\boldsymbol t}^*-{\boldsymbol t}^*_i,{\boldsymbol s}^*) = c_i({\boldsymbol t}^*,{\boldsymbol s}^*) \right] \geq c_i({\boldsymbol t}^*,{\boldsymbol s}^*) - c_i({\boldsymbol t}^*,{\boldsymbol s}), \label{eq:weirdCondition}
\end{align}
for all ${\boldsymbol s}\in\mathcal{S}_i({\boldsymbol t}^*)$. If
\[ h_b^i\left({\boldsymbol t}^*,{\boldsymbol s}^*\right) \leq c_i\left({\boldsymbol t}^*,{\boldsymbol s}^*\right), \qquad \text{and} \qquad  k_b^i\left({\boldsymbol t}^*,{\boldsymbol s}^*\right) \geq c_i\left({\boldsymbol t}^*,{\boldsymbol s}^*\right), \]
then
\begin{align*}
   &\lim\limits_{n\to\infty} \frac{1}{n}\log \mathbb{P}\left( Q_i^{(n)}   > b n \right) \\
   & \qquad\qquad = -\inf\limits_{{\boldsymbol t}\in\mathcal{T}_i} \sup\limits_{{\boldsymbol s}\in\mathcal{S}_i({\boldsymbol t})} \left\{ \frac{\Big[ b - \left(\mu_i-\overline\lambda_i\right) {\boldsymbol t}_i \Big]^2}{2\,\Var \Big( \bar A_i({\boldsymbol t}-{\boldsymbol t}_i,{\boldsymbol t}) \Big)} + \frac{\Big[ k_b^i({\boldsymbol t},{\boldsymbol s})- c_i({\boldsymbol t},{\boldsymbol s}) \Big]^2}{2\,\Var \Big( \bar A_i({\boldsymbol s},{\boldsymbol t}) \,\Big|\, \bar A_i({\boldsymbol t}-{\boldsymbol t}_i,{\boldsymbol t}) = b - \left(\mu_i-\overline\lambda_i\right) {\boldsymbol t}_i \Big)} \right\}.
  \end{align*}
\end{theorem}
The structure of the proof is the same as the proof of Theorem \ref{thm:acyclicLowerLD1}, and it is given in Appendix \ref{app:acyclicLowerLD3}.

\section{Example: equivalence to a single server queue}\label{sec:examples}

In this section we show that, if the input process is a multivariate fractional Brownian motion with non short-range dependence and non-negative correlation between its coordinates, and if the service rates are sufficiently large, then the large deviations behavior of any fixed queue in the network is the same as if all inputs to upstream queues were inputs to the queue itself. This phenomenon was also observed in \cite{MichelTandemPaper} for the second queue in a tandem, and here we generalize the conditions under which it occurs.

\subsection{Preliminaries on multivariate fractional Brownian motions}
Consider the case where the exogenous arrival process $A^{(n)}(\cdot)$ is a multivariate fractional Brownian motion (mfBm). Since each coordinate is a real valued fBm, for each $i\in\{1,\dots,k\}$, and for every $t<s<0$, we have
\[ {\mathbb C}{\rm ov}\left(\hat A^{(n)}_i(t),\,\, \hat A^{(n)}_i(s)\right) = \frac{\sigma_i^2}{2} \Big[ |t|^{2H_i} + |s|^{2H_i} - |s-t|^{2H_i} \Big], \]
where $H_i\in(0,1)$ is its Hurst index, and
\[ \sigma_i \triangleq \sqrt{ \Var \left(\hat A^{(n)}_i(1)\right) } \]
is its variance. Furthermore, it is known \cite{covariance} that, for each $i,j \in \{1,\dots,k\}$, and for every $t<s<0$, we have
\[ {\mathbb C}{\rm ov}\left(\hat A^{(n)}_i(t),\,\, \hat A^{(n)}_j(s)\right) =
\begin{cases}
\frac{\sigma_i \sigma_j}{2} \Big[ (\rho_{i,j}-\eta_{i,j})|t|^{H_i+H_j} + (\rho_{i,j}+\eta_{i,j})|s|^{H_i+H_j} - (\rho_{i,j}-\eta_{i,j})|s-t|^{H_i+H_j} \Big], \\
\hspace{7.5cm} \text{ if } H_i+H_j\neq 1, \\
\frac{\sigma_i \sigma_j}{2} \Big[ \rho_{i,j}\big(|t| + |s| - |s-t| \big) + \eta_{i,j}\big( s\log|s| - t\log|t| -(s-t)\log|s-t| \big) \Big], \\ \hspace{7.5cm} \text{ if } H_i+H_j= 1,
\end{cases} \]
where
\[ \rho_{i,j} \triangleq {\mathbb C}{\rm orr}\,\left(\hat A^{(n)}_i(1),\,\, \hat A^{(n)}_j(1)\right) \]
are their covariances, and $\eta_{i,j}=-\eta_{j,i}\in \mathbb{R}$ represents the inter-correlation in time between the two coordinates. Note that, contrary to the single-dimensional fBm, they need not be time-reversible. In particular, a mfBm is time-reversible if and only if $\eta_{i,j}=0$ for all $i,j$  \cite[Prop.\ 6]{basicProperties}. Moreover, the parameters $\eta_{i,j}$ have the following interpretation \cite{basicProperties}:
\begin{itemize}
  \item [(i)] If the one-dimensional fBm\,s are short-range dependent (i.e., if $H_i,H_j<1/2$), then they are either short-range interdependent if $\rho_{i,j}\neq 0$ or $\eta_{i,j}\neq 0$, or independent if $\rho_{i,j}=\eta_{i,j}=0$. This also holds when $H_i+H_j<1$, even if one of them is larger than or equal to $1/2$.
  \item [(ii)] If the one-dimensional fBm\,s are long-range dependent (i.e., if $H_i,H_j>1/2$), then they are either long-range interdependent if $\rho_{i,j}\neq 0$ or $\eta_{i,j}\neq 0$, or independent if $\rho_{i,j}=\eta_{i,j}=0$. This also holds when $H_i+H_j>1$, even if one of them is smaller than or equal to $1/2$.
  \item [(iii)] If the one-dimensional fBm\,s are Brownian motions (i.e., if $H_i=H_j=1/2$), then they are either long-range interdependent if $\eta_{i,j}\neq 0$, or independent if $\eta_{i,j}=0$. This also holds whenever $H_i+H_j=1$, even if neither of them are equal to $1/2$.
\end{itemize}

\subsection{Non-negatively correlated, non short-range dependent inputs}
We now present the main result of this section.

\begin{theorem}\label{thm:mfBmExample}
  Fix some node $i$. Suppose that $H_j=H \geq 1/2$, for all $j\in\{1,\dots,k\}$, that $\eta_{j,l}=0$, for all $j,l\in\{1,\dots,k\}$, and that $\rho_{j,l}\geq 0$, for all $j,l\in\{1,\dots,k\}$. Moreover, suppose that
  \begin{align}
    &\min\left\{  \mu_j - \lambda_j -\sum\limits_{l\in\mathcal{N}_{\rm in}(j)}  \mu_l p_{l,j} : j\neq i \right\} > \label{eq:simpleCondition} \\
    &\hspace{3mm}  \sup\limits_{\alpha\in (0,1)^{|\mathcal{P}_2(i)|}} \left\{ \frac{ \sum\limits_{r\in\mathcal{P}_2(i)} \left( \sigma_{r_1}\sigma_i\rho_{r_1,i} + \sum\limits_{r'\in\mathcal{P}_2(i)} \sigma_{r_1}\sigma_{r'_1}\rho_{r_1,r'_1} \Pi_{r'}\right) \left[\left(\alpha_r\right)^{2H}+ 1 - \left( 1- \alpha_r\right)^{2H}\right]\Pi_r }{ \left( \sum\limits_{r\in\mathcal{P}_2(i)} \alpha_r \Pi_r \right) } \right\} \left(\frac{\mu_i-\overline{\lambda}_i}{2H \overline{\sigma}_i^2}\right), \nonumber
  \end{align}
  where
  \[ \overline{\sigma}_i^2 \triangleq \sigma_i^2 + \sum\limits_{r\in\mathcal{P}_2(i)} \left( 2 \sigma_{r_1}\sigma_i\rho_{r_1,i} + \sum\limits_{r'\in\mathcal{P}_2(i)} \sigma_{r_1}\sigma_{r'_1}\rho_{r_1,r'_1} \Pi_{r'}\right)\Pi_r. \]
  Then, for every $b>0$,
  \[ - \lim\limits_{n\to\infty} \frac{1}{n}\log \mathbb{P}\left( Q^{(n)}_i > nb \right)  = \frac{1}{2\overline{\sigma}_i^2}\left( \frac{b}{1-H} \right)^{2-2H}\left( \frac{\mu_i - \overline{\lambda}_i}{H} \right)^{2H}. \]
\end{theorem}
The proof is given in Appendix \ref{app:mfBmExample}, and amounts to  checking that Theorem \ref{thm:acyclicLowerLD1} applies in this case, to then compute the exact decay rate.

\begin{remark} {\em
  Note that this decay rate is the same as the one that we would obtain in a single-server queue with processing rate $\mu_i$ and input
  \[ \sum\limits_{r\in\mathcal{P}_1(i)} A^{(n)}_{r_1}(\cdot)\Pi_r. \]
  This means that, under the assumptions of Theorem \ref{thm:mfBmExample},  in this regime the queues upstream of node $i$ are `transparent'. In particular, this implies that the most likely overflow path is the one where all upstream queues are empty.
}\end{remark}

\begin{remark} {\em
  In the case of a pair of queues tandem with arrivals only to the first queue, the condition in   \eqref{eq:simpleCondition} is the same as the one obtained in the analysis of the tandem queues done in \cite{MichelTandemPaper}.
}\end{remark}

\section{Conclusions}\label{sec:conclusion}
We have considered an acyclic network of queues with (possibly correlated) Gaussian inputs and static routing, and characterized the large deviations behavior of the steady-state queue length in each queue of the network. We achieved this by defining an appropriate multi-dimensional Reproducing Kernel Hilbert Space, and using Schilder's theorem to obtain lower and upper bounds for the asymptotic exponential decay rate. This generalizes previous results, which focused on isolated queues and two-queue tandem systems (with arrivals only to the first queue).

While the results that we obtain are quite general both in terms of the network structure and in terms of the correlation structure among the arrival processes to the different nodes, there are still interesting open problems. For instance:
\begin{itemize}
  \item [(i)] While we considered essentially only single-class traffic with a deterministic split of the work departing from each server, it would be interesting to extend our results to multi-class networks, where the servers are shared by using, for example, the Generalized Processor Sharing discipline \cite{GPSQueues}.
  \item [(ii)] While we only obtained large-deviations results for each queue separately, it would be interesting to obtain similar results for the joint queue lengths.
\end{itemize}

\appendix

\section{Proof of Lemma \ref{lem:inputFormula}}\label{app:proofInputFormula}

We prove this by induction in the maximum length of paths that end in node $i$. Suppose that the maximum length is one. Then, $\mathcal{P}_2(i)=\emptyset$ and thus
\[ I^{(n)}_i(t,0) = A^{(n)}_i(t,0). \]

Now suppose that  \eqref{eq:inputProcess} holds for all nodes $j$ such that the maximum length of paths that end in $j$ is at most one less than the maximum lengths of paths that end in node $i$. Recall that
\begin{align*}
 D^{(n)}_j(t,0) &= Q^{(n)}_j(t) + I^{(n)}_j(t,0) - Q^{(n)}_j(0),\\
  Q^{(n)}_j(t)& = \sup\limits_{s<t} \left\{ I^{(n)}_j(s,t) - n\mu_j (t-s) \right\}.
\end{align*}
Combining the last two equations, we obtain that $I^{(n)}_i(t,0)$ equals
\begin{align*}
A^{(n)}_i&(t,0) + \sum\limits_{j\in\mathcal{N}_{\rm in}(i)} p_{j,i} D^{(n)}_j(t,0) \\
  &= A^{(n)}_i(t,0) + \sum\limits_{j\in\mathcal{N}_{\rm in}(i)} p_{j,i} \left[ \sup\limits_{t_j<t} \left\{ I^{(n)}_j(t_j,t) - n\mu_j (t-t_j) \right\} + I^{(n)}_j(t,0) - \sup\limits_{s_j<0} \left\{ I^{(n)}_j(s_j,0) + n\mu_j s_j \right\} \right] \\
  &= A^{(n)}_i(t,0) + \sum\limits_{j\in\mathcal{N}_{\rm in}(i)} p_{j,i} \left[ \sup\limits_{t_j<t} \left\{ I^{(n)}_j(t_j,0) - n\mu_j (t-t_j) \right\} - \sup\limits_{s_j<0} \left\{ I^{(n)}_j(s_j,0) + n\mu_j s_j \right\} \right].
\end{align*}
Since all $j$ are inbound neighbors of $i$, and the graph is acyclic, the maximum lengths of paths that end in nodes $j$ are at most one less than the maximum length of paths that end in node $i$. Then, using the inductive hypothesis on the input processes $I^{(n)}_j(t_j,0)$, $I^{(n)}_i(t,0)$ equals
$A^{(n)}_i(t,0)$ increased by
\begin{align*}
  \sum\limits_{j\in\mathcal{N}_{\rm in}(i)} &p_{j,i} \left[ \sup\limits_{t_j<t} \left\{ A^{(n)}_j(t_j,0) +  \sup\limits_{{\boldsymbol t}\in\mathcal{T}_j(t_j)} \left\{ \sum\limits_{r\in\mathcal{P}_2(j)}  \left[ A^{(n)}_{r_1}({\boldsymbol t}_{r},0)+ n\mu_{r_1}({\boldsymbol t}_{r}-{\boldsymbol t}_{r_+}) \right] \Pi_r \right\} \right.\right. \\
   &\qquad\qquad    \left.- \sup\limits_{{\boldsymbol s}\in\mathcal{T}_j(0)} \left\{ \sum\limits_{r\in\mathcal{P}_2(j)} \left[ A^{(n)}_{r_1}({\boldsymbol s}_{r},0)+ n\mu_{r_1}({\boldsymbol s}_{r}-{\boldsymbol s}_{r_+}) \right] \Pi_r \right\} - n\mu_j (t-t_j) \right\} \\
  &\qquad - \sup\limits_{s_j<0} \left\{ A^{(n)}_j(s_j,0) +  \sup\limits_{{\boldsymbol t}\in\mathcal{T}_j(s_j)} \left\{ \sum\limits_{r\in\mathcal{P}_2(j)}  \left[ A^{(n)}_{r_1}({\boldsymbol t}_{r},0)+ n\mu_{r_1}({\boldsymbol t}_{r}-{\boldsymbol t}_{r_+}) \right] \Pi_r \right\} \right. \\
   &\qquad\qquad     \left. \left.- \sup\limits_{{\boldsymbol s}\in\mathcal{T}_j(0)} \left\{ \sum\limits_{r\in\mathcal{P}_2(j)} \left[ A^{(n)}_{r_1}({\boldsymbol s}_{r},0)+ n\mu_{r_1}({\boldsymbol s}_{r}-{\boldsymbol s}_{r_+}) \right] \Pi_r \right\} + n\mu_j s_j \right\} \right] =\\
    \sum\limits_{j\in\mathcal{N}_{\rm in}(i)} &p_{j,i} \left[ \sup\limits_{t_j<t} \left\{ A^{(n)}_j(t_j,0) +  \sup\limits_{{\boldsymbol t}\in\mathcal{T}_j(t_j)} \left\{ \sum\limits_{r\in\mathcal{P}_2(j)}  \left[ A^{(n)}_{r_1}({\boldsymbol t}_{r},0)+ n\mu_{r_1}({\boldsymbol t}_{r}-{\boldsymbol t}_{r_+}) \right] \Pi_r \right\} - n\mu_j (t-t_j) \right\} \right. \\
  &\qquad\qquad  \left. - \sup\limits_{s_j<0} \left\{ A^{(n)}_j(s_j,0) +  \sup\limits_{{\boldsymbol t}\in\mathcal{T}_j(s_j)} \left\{ \sum\limits_{r\in\mathcal{P}_2(j)}  \left[ A^{(n)}_{r_1}({\boldsymbol t}_{r},0)+ n\mu_{r_1}({\boldsymbol t}_{r}-{\boldsymbol t}_{r_+}) \right] \Pi_r \right\}  + n\mu_j s_j \right\} \right].
\end{align*}
After renaming the variables for ease of exposition, we obtain $I^{(n)}_i(t,0)$ equals
$A^{(n)}_i(t,0)$ increased by
\begin{align*}
  &\sum\limits_{j\in\mathcal{N}_{\rm in}(i)} p_{j,i} \left[ \sup\limits_{t_j<t} \left\{ A^{(n)}_j(t_j,0) +  \sup\limits_{{\boldsymbol t}^{(j)}\in\mathcal{T}_j(t_j)} \left\{ \sum\limits_{r\in\mathcal{P}_2(j)}  \left[ A^{(n)}_{r_1}({\boldsymbol t}^{(j)}_r,0)+ n\mu_{r_1}({\boldsymbol t}^{(j)}_r-{\boldsymbol t}^{(j)}_{r_+}) \right] \Pi_r \right\} - n\mu_j (t-t_j) \right\} \right. \\
  &\qquad \left. - \sup\limits_{s_j<0} \left\{ A^{(n)}_j(s_j,0) +  \sup\limits_{{\boldsymbol s}^{(j)}\in\mathcal{T}_j(s_j)} \left\{ \sum\limits_{r\in\mathcal{P}_2(j)}  \left[ A^{(n)}_{r_1}({\boldsymbol s}^{(j)}_r,0)+ n\mu_{r_1}({\boldsymbol s}^{(j)}_r-{\boldsymbol s}^{(j)}_{r_+}) \right] \Pi_r \right\}  + n\mu_j s_j \right\} \right] \\
  &= \sup\limits_{t_j<t} \left\{ \sup\limits_{{\boldsymbol t}^{(j)}\in\mathcal{T}_j(t_j)} \left\{ \sum\limits_{j\in\mathcal{N}_{\rm in}(i)} p_{j,i} \left[ A^{(n)}_j(t_j,0) +  \sum\limits_{r\in\mathcal{P}_2(j)}  \left[ A^{(n)}_{r_1}({\boldsymbol t}^{(j)}_r,0)+ n\mu_{r_1}({\boldsymbol t}^{(j)}_r-{\boldsymbol t}^{(j)}_{r_+}) \right] \Pi_r - n\mu_j (t-t_j) \right] \right\} \right\} \\
  &\qquad  - \sup\limits_{s_j<0} \left\{ \sup\limits_{{\boldsymbol s}^{(j)}\in\mathcal{T}_j(s_j)} \left\{ \sum\limits_{r\in\mathcal{P}_2(j)} \left[ A^{(n)}_j(s_j,0) +  \sum\limits_{r\in\mathcal{P}_2(j)}  \left[ A^{(n)}_{r_1}({\boldsymbol s}^{(j)}_r,0)+ n\mu_{r_1}({\boldsymbol s}^{(j)}_r-{\boldsymbol s}^{(j)}_{r_+}) \right] \Pi_r + n\mu_j s_j \right] \right\} \right\} \\
  &=\sup\limits_{{\boldsymbol t}\in\mathcal{T}_i(t)} \left\{ \sum\limits_{r\in\mathcal{P}_2(i)}  \left[ A^{(n)}_{r_1}({\boldsymbol t}_{r},0)+ n\mu_{r_1}({\boldsymbol t}_{r}-{\boldsymbol t}_{r_+}) \right] \Pi_r \right\}  - \sup\limits_{{\boldsymbol s}\in\mathcal{T}_i(0)} \left\{ \sum\limits_{r\in\mathcal{P}_2(i)} \left[ A^{(n)}_{r_1}({\boldsymbol s}_{r},0)+ n\mu_{r_1}({\boldsymbol s}_{r}-{\boldsymbol s}_{r_+}) \right] \Pi_r \right\}.
\end{align*}

\section{Proof of Theorem \ref{thm:probAsFunctOfInputs}}\label{app:probAsFunctOfInputs}

By Reich's formula, we have
\begin{align*}
  \mathbb{P}\left(Q^{(n)}_i > nb\right) &= \mathbb{P}\left(\sup\limits_{{\boldsymbol t}_i< 0} \left\{ I^{(n)}_i({\boldsymbol t}_i,0) + n\mu_i {\boldsymbol t}_i \right\} > nb\right).
\end{align*}
By Lemma \ref{lem:inputFormula}, we obtain
\begin{align*}
  \mathbb{P}\left(Q^{(n)}_i > nb\right) &= \mathbb{P}\left(\sup\limits_{{\boldsymbol t}_i< 0} \left\{ A^{(n)}_i({\boldsymbol t}_i,0) +  \sup\limits_{{\boldsymbol t}\in\mathcal{T}_i({\boldsymbol t}_i)} \left\{ \sum\limits_{r\in\mathcal{P}_2(i)}  \left[ A^{(n)}_{r_1}({\boldsymbol t}_{r},0)+ n\mu_{r_1}({\boldsymbol t}_{r}-{\boldsymbol t}_{r_+}) \right] \Pi_r \right.\right.\right. \\
 &\qquad\qquad\qquad\qquad \left.\left.\left. - \sup\limits_{{\boldsymbol s}\in\mathcal{S}_i({\boldsymbol t})} \left\{ \sum\limits_{r\in\mathcal{P}_2(i)} \left[ A^{(n)}_{r_1}({\boldsymbol s}_{r},0)+ n\mu_{r_1}({\boldsymbol s}_{r}-{\boldsymbol s}_{r_+}) \right] \Pi_r \right\} \right\} + n\mu_i {\boldsymbol t}_i \right\} > nb\right) \\
&= \mathbb{P}\left( \exists\, {\boldsymbol t}_i<0,\,\, {\boldsymbol t}\in\mathcal{T}_i({\boldsymbol t}_i) : \forall\, {\boldsymbol s}\in\mathcal{S}_i({\boldsymbol t}) : \frac{1}{n} \left( A^{(n)}_i({\boldsymbol t}_i,0) + \sum\limits_{r\in\mathcal{P}_2(i)} \left[ A^{(n)}_{r_1}({\boldsymbol t}_{r},0) - A^{(n)}_{r_1}({\boldsymbol s}_{r},0) \right] \Pi_r \right) \right. \\
 &\qquad\qquad\qquad\qquad\qquad \left. > b - \mu_i {\boldsymbol t}_i -  \sum\limits_{r\in\mathcal{P}_2(i)} \Big[\mu_{r_1}\big({\boldsymbol t}_{r}-{\boldsymbol s}_{r}\big)-\mu_{r_1}\big({\boldsymbol t}_{r_+}-{\boldsymbol s}_{r_+}\big)\Big] \Pi_r \right) \\
 &= \mathbb{P}\left( \exists\, {\boldsymbol t}\in\mathcal{T}_i : \forall\, {\boldsymbol s}\in\mathcal{S}_i({\boldsymbol t}) : \frac{1}{n} \left( -A^{(n)}_i({\boldsymbol t}_i) - \sum\limits_{r\in\mathcal{P}_2(i)} \left[ A^{(n)}_{r_1}({\boldsymbol t}_{r}) - A^{(n)}_{r_1}({\boldsymbol s}_{r}) \right] \Pi_r \right) \right. \\
 &\qquad\qquad\qquad\qquad\qquad \left. > b - \mu_i {\boldsymbol t}_i -  \sum\limits_{r\in\mathcal{P}_2(i)} \Big[\mu_{r_1}\big({\boldsymbol t}_{r}-{\boldsymbol s}_{r}\big)-\mu_{r_1}\big({\boldsymbol t}_{r_+}-{\boldsymbol s}_{r_+}\big)\Big] \Pi_r \right) \\
 &= \mathbb{P}\left( \frac{A^{(n)}(\cdot)-n\lambda\, \cdot\,}{n} \in \tilde{\mathcal{E}}^i(b) \right),
\end{align*}
where
\begin{align*}
 &\tilde{\mathcal{E}}^i(b) \triangleq \Big\{ f\in\Omega^k: \exists\, {\boldsymbol t}\in\mathcal{T}_i : \forall\, {\boldsymbol s}\in\mathcal{S}_i({\boldsymbol t}),\,\, - f_i({\boldsymbol t}_i) - \sum\limits_{r\in\mathcal{P}_2(i)} \Big[f_{r_1}({\boldsymbol t}_{r}) - f_{r_1}({\boldsymbol s}_{r})\Big] \Pi_r \\
 &\qquad\qquad\qquad\qquad\qquad > b - (\mu_i-\lambda_i) {\boldsymbol t}_i -  \sum\limits_{r\in\mathcal{P}_2(i)} \Big[\big(\mu_{r_1}-\lambda_{r_1}\big)\big({\boldsymbol t}_{r}-{\boldsymbol s}_{r}\big)-\mu_{r_1}\big({\boldsymbol t}_{r_+}-{\boldsymbol s}_{r_+}\big)\Big] \Pi_r \Big\}.
\end{align*}
Since the centered Gaussian processes are symmetric, we have
\begin{align*}
  \mathbb{P}\left(Q^{(n)}_i > nb\right) &= \mathbb{P}\left( \frac{A^{(n)}(\cdot)-n\lambda\, \cdot\,}{n} \in \mathcal{E}^i(b) \right),
\end{align*}
where
\begin{align*}
 &\mathcal{E}^i(b) \triangleq \left\{ f\in\Omega^k: \exists\, {\boldsymbol t}\in\mathcal{T}_i : \forall\, {\boldsymbol s}\in\mathcal{S}_i({\boldsymbol t}),\,\, f_i({\boldsymbol t}_i) +\sum\limits_{r\in\mathcal{P}_2(i)} \Big[f_{r_1}({\boldsymbol t}_{r}) - f_{r_1}({\boldsymbol s}_{r})\Big] \Pi_r \right.\\
 &\qquad\qquad\qquad\qquad \left. > b - (\mu_i-\lambda_i) {\boldsymbol t}_i -  \sum\limits_{r\in\mathcal{P}_2(i)} \Big[\big(\mu_{r_1}-\lambda_{r_1}\big)\big({\boldsymbol t}_{r}-{\boldsymbol s}_{r}\big)-\mu_{r_1}\big({\boldsymbol t}_{r_+}-{\boldsymbol s}_{r_+}\big)\Big] \Pi_r \right\}.
\end{align*}
Finally, rearranging terms, and using that ${\boldsymbol s}_i=0$, we obtain
\begin{align*}
 &\mathcal{E}_i(b) \triangleq \left\{ f\in\Omega^k: \exists\, {\boldsymbol t}\in\mathcal{T}_i : \forall\, {\boldsymbol s}\in\mathcal{S}_i({\boldsymbol t}),\,\, f_i({\boldsymbol t}_i) + \sum\limits_{r\in\mathcal{P}_2(i)} \Big[f_{r_1}({\boldsymbol t}_{r}) - f_{r_1}({\boldsymbol s}_{r})\Big] \Pi_r \right. \\
 &\qquad\qquad\qquad\qquad\qquad\qquad\qquad \left. > b - \sum\limits_{r\in\mathcal{P}_1(i)} \left[ \left(\mu_{r_1}-\lambda_{r_1} - \sum\limits_{j\in\mathcal{N}_{\rm in}(r_1)} \mu_j p_{j,r_1}\right) \big({\boldsymbol t}_{r}-{\boldsymbol s}_{r}\big)\right] \Pi_r \right\}.
\end{align*}

\section{Proof of Theorem \ref{thm:acyclicUpperLD}}\label{app:acyclicUpperLD}
The proof consists of two steps. First, we decompose the event $\mathcal{E}_i(b)$ given in Theorem \ref{thm:probAsFunctOfInputs} as a union of intersections of simpler events that only involve the sample paths at fixed times, and we majorize the probability of the intersection by the probability of the least likely one (Lemma \ref{lem:miniLowerBound}). Then, we use Cram\'er's theorem to obtain the decay rate of the least likely of these simpler events by solving the additional quadratic optimization problem that arises by its application (Lemma \ref{lem:computeI}).

\begin{lemma}\label{lem:miniLowerBound}
  We have
\[  \inf\limits_{f\in \mathcal{E}_i(b)} \big\{ \mathbb{I}(f) \big\} \geq  \inf\limits_{{\boldsymbol t}\in\mathcal{T}_i} \sup\limits_{{\boldsymbol s}\in\mathcal{S}_i({\boldsymbol t})} \inf\limits_{f\in\mathcal{U}_{{\boldsymbol t},{\boldsymbol s}}} \big\{ \mathbb{I}(f) \big\}, \]
where
\begin{align*}
 \mathcal{U}_{{\boldsymbol t},{\boldsymbol s}} \triangleq \Big\{ f\in\Omega^k :& f_i({\boldsymbol t}_i) + \sum\limits_{r\in\mathcal{P}_2(i)} \Big[f_{r_1}({\boldsymbol t}_{r}) - f_{r_1}({\boldsymbol t}_{r}-{\boldsymbol t}_i)\Big] \Pi_r \geq b - \left(\mu_i-\overline\lambda_i\right) {\boldsymbol t}_i \:\:\text{and} \\
 &   f_i({\boldsymbol t}_i) + \sum\limits_{r\in\mathcal{P}_2(i)} \Big[f_{r_1}({\boldsymbol t}_{r}) - f_{r_1}({\boldsymbol s}_{r})\Big] \Pi_r \geq b - (\mu_i-\overline\lambda_i) {\boldsymbol t}_i - c_i({\boldsymbol t},{\boldsymbol s}) \Big\}.
\end{align*}
\end{lemma}

\begin{remark} {\em
Note that the first condition in the definition of the set $\mathcal{U}_{{\boldsymbol t},{\boldsymbol s}}$ is the same as the second one, but with ${\boldsymbol s}_r={\boldsymbol t}_r-{\boldsymbol t}_i$, for all $r\in\mathcal{P}_2(i)$. This generalizes Theorem 3.2 in \cite{MichelTandemPaper}, where an appropriate $\mathcal{U}_{{\boldsymbol t},{\boldsymbol s}}$ is defined by having the first condition being the same as the second one but with ${\boldsymbol s}_r=0$, for all $r\in\mathcal{P}_2(i)$. In the case of a tandem with arrivals only to the first queue, both definitions are equivalent.
}\end{remark}

\begin{proof}
Recall that
\begin{align*}
 &\mathcal{E}^i(b) \triangleq \left\{ f\in\Omega^k: \exists\, {\boldsymbol t}\in\mathcal{T}_i : \forall\, {\boldsymbol s}\in\mathcal{S}_i({\boldsymbol t}),\,\, f_i({\boldsymbol t}_i) +\sum\limits_{r\in\mathcal{P}_2(i)} \Big[f_{r_1}({\boldsymbol t}_{r}) - f_{r_1}({\boldsymbol s}_{r})\Big] \Pi_r > b - (\mu_i-\overline\lambda_i) {\boldsymbol t}_i - c_i({\boldsymbol t},{\boldsymbol s}) \right\}.
\end{align*}
Thus
\[ \mathcal{E}^i(b) = \bigcup\limits_{{\boldsymbol t}\in\mathcal{T}_i} \bigcap\limits_{{\boldsymbol s}\in\mathcal{S}_i({\boldsymbol t})} \mathcal{E}_{{\boldsymbol t},{\boldsymbol s}}, \]
where
\begin{align*}
 &\mathcal{E}_{{\boldsymbol t},{\boldsymbol s}} \triangleq \left\{ f\in\Omega^k:  f_i({\boldsymbol t}_i) +\sum\limits_{r\in\mathcal{P}_2(i)} \Big[f_{r_1}({\boldsymbol t}_{r}) - f_{r_1}({\boldsymbol s}_{r})\Big] \Pi_r > b - (\mu_i-\overline\lambda_i) {\boldsymbol t}_i - c_i({\boldsymbol t},{\boldsymbol s}) \right\}.
\end{align*}
Then, we have
\begin{equation}\label{eq:intermediateIInequality}
 \inf\limits_{f\in \mathcal{E}_i(b)} \big\{ \mathbb{I}(f) \big\} = \inf\limits_{{\boldsymbol t}\in\mathcal{T}_i} \inf\limits_{f\in \bigcap\limits_{{\boldsymbol s}\in\mathcal{S}_i({\boldsymbol t})} \mathcal{E}_{{\boldsymbol t},{\boldsymbol s}}} \big\{ \mathbb{I}(f) \big\}.
\end{equation}
Now fix ${\boldsymbol t}\in\mathcal{T}_i$, and consider the innermost infimum. Since $f$ is continuous, then
\[ f_i({\boldsymbol t}_i) + \sum\limits_{r\in\mathcal{P}_2(i)} \Big[f_{r_1}({\boldsymbol t}_{r}) - f_{r_1}({\boldsymbol s}_{r})\Big] \Pi_r > b - (\mu_i-\overline\lambda_i) {\boldsymbol t}_i - c_i({\boldsymbol t},{\boldsymbol s}) \]
for all ${\boldsymbol s}\in\mathcal{S}_i({\boldsymbol t})$ implies
\[ f_i({\boldsymbol t}_i) + \sum\limits_{r\in\mathcal{P}_2(i)} \Big[f_{r_1}({\boldsymbol t}_{r}) - f_{r_1}({\boldsymbol s}_{r})\Big] \Pi_r \geq b - (\mu_i-\overline\lambda_i) {\boldsymbol t}_i - c_i({\boldsymbol t},{\boldsymbol s}) \]
for all ${\boldsymbol s}\in\overline{\mathcal{S}_i({\boldsymbol t})}$. Hence
\[ \bigcap\limits_{{\boldsymbol s}\in\mathcal{S}_i({\boldsymbol t})} \mathcal{E}_{{\boldsymbol t},{\boldsymbol s}} \subset \bigcap\limits_{{\boldsymbol s}\in\overline{\mathcal{S}_i({\boldsymbol t})}} \mathcal{U}_{{\boldsymbol t},{\boldsymbol s}} \subset \mathcal{U}_{t_t,{\boldsymbol t},{\boldsymbol r}}, \]
for all ${\boldsymbol r}\in\mathcal{S}_i({\boldsymbol t})$, and thus
\[ \inf\limits_{f\in \bigcap\limits_{{\boldsymbol s}\in\mathcal{S}_i({\boldsymbol t})} \mathcal{E}_{{\boldsymbol t},{\boldsymbol s}}} \big\{ \mathbb{I}(f) \big\} \geq \inf\limits_{f\in\mathcal{U}_{{\boldsymbol t},{\boldsymbol r}}} \big\{ \mathbb{I}(f) \big\}. \]
Therefore,
\[ \inf\limits_{f\in \bigcap\limits_{{\boldsymbol s}\in\mathcal{S}_i({\boldsymbol t})} \mathcal{E}_{{\boldsymbol t},{\boldsymbol s}}} \big\{ \mathbb{I}(f) \big\} \geq \sup\limits_{{\boldsymbol r}\in\mathcal{S}_i({\boldsymbol t})} \inf\limits_{f\in\mathcal{U}_{{\boldsymbol t},{\boldsymbol r}}} \big\{ \mathbb{I}(f) \big\}. \]
Combining this with  \eqref{eq:intermediateIInequality} completes the proof.
\end{proof}

\begin{remark} {\em
  Note that, by taking the supremum over all ${\boldsymbol r}\in\mathcal{S}_i({\boldsymbol t})$ at the end of the proof, we are essentially upper bounding the probability of an intersection with the probability of the least likely event.
}\end{remark}

While we have made progress towards obtaining the desired expression for the limiting overflow probability, the expression in Lemma \ref{lem:miniLowerBound} still depends on the rate function $\mathbb{I}$. We now proceed to compute this simpler expression.

\begin{lemma}\label{lem:computeI}
Under Assumption \ref{ass:acyclic}, for ${\boldsymbol t}\in\mathcal{T}_i$, and ${\boldsymbol s}\in\mathcal{S}_i({\boldsymbol t})$, we have
\begin{equation*}
\inf\limits_{f\in\mathcal{U}_{{\boldsymbol t},{\boldsymbol s}}} \big\{ \mathbb{I}(f) \big\} =
\begin{cases}
\frac{\Big[ b - \left(\mu_i-\overline\lambda_i\right) {\boldsymbol t}_i \Big]^2}{2\,\Var \Big( \bar A_i({\boldsymbol t}-{\boldsymbol t}_i,{\boldsymbol t}) \Big)}, & \begin{array}{l}\text{if } k_b^i({\boldsymbol t},{\boldsymbol s}) < c_i({\boldsymbol t},{\boldsymbol s}),\\ \quad \text{or} \quad {\boldsymbol s}={\boldsymbol t}-{\boldsymbol t}_i,\end{array} \\
\frac{\Big[b - \left(\mu_i-\overline\lambda_i\right) {\boldsymbol t}_i - c_i({\boldsymbol t},{\boldsymbol s}) \Big]^2}{2\,\Var \Big( \bar A_i({\boldsymbol s},{\boldsymbol t}) \Big)}, & \begin{array}{l}\text{if } h_b^i({\boldsymbol t},{\boldsymbol s})> c_i({\boldsymbol t},{\boldsymbol s}),\end{array} \\
\frac{\Big[ b - \left(\mu_i-\overline\lambda_i\right) {\boldsymbol t}_i \Big]^2}{2\,\Var \Big( \bar A_i({\boldsymbol t}-{\boldsymbol t}_i,{\boldsymbol t}) \Big)} + \frac{\Big[ k_b^i({\boldsymbol t},{\boldsymbol s})- c_i({\boldsymbol t},{\boldsymbol s}) \Big]^2}{2\,\Var \Big( \bar A_i({\boldsymbol s},{\boldsymbol t}) \,\Big|\, \bar A_i({\boldsymbol t}-{\boldsymbol t}_i,{\boldsymbol t}) = b - \left(\mu_i-\overline\lambda_i\right) {\boldsymbol t}_i \Big)}, & \begin{array}{l}\text{otherwise.}\end{array}
\end{cases}
\end{equation*}
\end{lemma}
\begin{proof}
Recall that
\[\mathbb{P}\left( \frac{A^{(n)}(\cdot)-n\lambda\, \cdot\,}{n} \in \mathcal{U}_{{\boldsymbol t},{\boldsymbol s}} \right)\]
can be rewritten as
\begin{align}
 \mathbb{P}\Bigg( \frac{1}{n}& \left( A^{(n)}_i({\boldsymbol t}_i) + \sum\limits_{r\in\mathcal{P}_2(i)} \left[ A^{(n)}_{r_1}({\boldsymbol t}_{r}) - A^{(n)}_{r_1}({\boldsymbol t}_{r}-{\boldsymbol t}_i) \right] \Pi_r \right) \geq b - \mu_i {\boldsymbol t}_i \qquad \text{and}  \label{eq:basicProbability} \\
   \frac{1}{n}& \left( A^{(n)}_i({\boldsymbol t}_i) + \sum\limits_{r\in\mathcal{P}_2(i)} \left[ A^{(n)}_{r_1}({\boldsymbol t}_{r}) - A^{(n)}_{r_1}({\boldsymbol s}_{r}) \right] \Pi_r \right) \geq b - \mu_i {\boldsymbol t}_i -  \sum\limits_{r\in\mathcal{P}_2(i)} \mu_{r_1}\Big[\big({\boldsymbol t}_{r}-{\boldsymbol t}_{r_+}\big)-\big({\boldsymbol s}_{r}-{\boldsymbol s}_{r_+}\big)\Big] \Pi_r \Bigg). \nonumber
\end{align}
Since this probability only depends on the state of the trajectories at fixed points in time, that is, only depends on a finite set of Gaussian random variables, it follows that $\mathcal{U}_{{\boldsymbol t},{\boldsymbol s}}$ is a $\mathbb{I}$-continuity set,
%{\color{red}[WE AGREED TO OMIT THIS, BUT WE CAN INCLUDE A PROOF IF IT FEELS BETTER THAT WAY]},
and thus Schilder's theorem implies that
\begin{equation}\label{eq:CramerBasic}
 -\lim\limits_{n\to\infty} \frac{1}{n} \log \mathbb{P}\left( \frac{A^{(n)}(\cdot)-n\lambda\, \cdot\,}{n} \in \mathcal{U}_{{\boldsymbol t},{\boldsymbol s}} \right)  = \inf\limits_{f\in\mathcal{U}_{{\boldsymbol t},{\boldsymbol s}}} \big\{ \mathbb{I}(f) \big\}.
\end{equation}
We now proceed to compute the left-hand side.

First, consider the exceptional case where ${\boldsymbol s}={\boldsymbol t}-{\boldsymbol t}_i$. Substituting this in   \eqref{eq:basicProbability}, we get
\begin{align}
 &\mathbb{P}\left( \frac{A^{(n)}(\cdot)-n\lambda\, \cdot\,}{n} \in \mathcal{U}_{{\boldsymbol t},{\boldsymbol s}} \right) = \mathbb{P}\left( \frac{1}{n} \left( A^{(n)}_i({\boldsymbol t}_i) + \sum\limits_{r\in\mathcal{P}_2(i)} \left[ A^{(n)}_{r_1}({\boldsymbol t}_{r}) - A^{(n)}_{r_1}({\boldsymbol t}_{r}-{\boldsymbol t}_i) \right] \Pi_r \right) \geq b - \mu_i {\boldsymbol t}_i \right). \label{eq:exceptionalProbability}
\end{align}
Moreover, by Cram\'er's theorem, we have that
\[ -\lim\limits_{n\to\infty} \frac{1}{n} \log \mathbb{P}\left( \frac{1}{n} \left( A^{(n)}_i({\boldsymbol t}_i) + \sum\limits_{r\in\mathcal{P}_2(i)} \left[ A^{(n)}_{r_1}({\boldsymbol t}_{r}) - A^{(n)}_{r_1}({\boldsymbol t}_{r}-{\boldsymbol t}_i) \right] \Pi_r \right) \geq b - \mu_i {\boldsymbol t}_i \right)  = \frac{\Big[ b - \left(\mu_i-\overline\lambda_i\right) {\boldsymbol t}_i \Big]^2}{2\,\Var \left( \bar A_i({\boldsymbol t}-{\boldsymbol t}_i,{\boldsymbol t}) \right)}. \]
Combining this with   \eqref{eq:CramerBasic} and \eqref{eq:exceptionalProbability}, we obtain
\[ \inf\limits_{f\in\mathcal{U}_{{\boldsymbol t},{\boldsymbol s}}} \big\{ \mathbb{I}(f) \big\} = \frac{\Big[ b - \left(\mu_i-\overline\lambda_i\right) {\boldsymbol t}_i \Big]^2}{2\,\Var \left( \bar A_i({\boldsymbol t}-{\boldsymbol t}_i,{\boldsymbol t}) \right)}. \]

Now consider the case when ${\boldsymbol s}\neq{\boldsymbol t}-{\boldsymbol t}_i$. By the multivariate version of Cram\'er Theorem, we have that
\begin{align*}
 \lim\limits_{n\to\infty} \frac{1}{n} \log \mathbb{P}\Bigg( &\frac{1}{n} \left( A^{(n)}_i({\boldsymbol t}_i) + \sum\limits_{r\in\mathcal{P}_2(i)} \left[ A^{(n)}_{r_1}({\boldsymbol t}_{r}) - A^{(n)}_{r_1}({\boldsymbol s}_{r}) \right] \Pi_r \right) \geq \\& \qquad  \qquad b - \mu_i {\boldsymbol t}_i - \sum\limits_{r\in\mathcal{P}_2(i)} \mu_{r_1}\Big[\big({\boldsymbol t}_{r}-{\boldsymbol t}_{r_+}\big)-\big({\boldsymbol s}_{r}-{\boldsymbol s}_{r_+}\big)\Big] \Pi_r, \\
     & \frac{1}{n} \left( A^{(n)}_i({\boldsymbol t}_i) + \sum\limits_{r\in\mathcal{P}_2(i)} \left[ A^{(n)}_{r_1}({\boldsymbol t}_{r}) - A^{(n)}_{r_1}({\boldsymbol t}_{r}-{\boldsymbol t}_i) \right] \Pi_r \right) \geq b - \mu_i {\boldsymbol t}_i \Bigg) \\
 &\qquad \qquad = \inf\Big\{ \Lambda_{{\boldsymbol t},{\boldsymbol s}}(y,z) : y\geq b - \left(\mu_i-\overline\lambda_i\right) {\boldsymbol t}_i;\,\, z\geq b - (\mu_i-\overline\lambda_i) {\boldsymbol t}_i - c_i({\boldsymbol t},{\boldsymbol s}) \Big\},
\end{align*}
where
\begin{equation}\label{eq:Lambda}
 \Lambda_{{\boldsymbol t},{\boldsymbol s}}(y,z) \triangleq \frac{1}{2}\left(y,\, z\right)\begin{pmatrix}
 \Var \Big(\bar A_i({\boldsymbol t}-{\boldsymbol t}_i,{\boldsymbol t}) \Big) & {\mathbb C}{\rm ov}\left( \bar A_i({\boldsymbol t}-{\boldsymbol t}_i,{\boldsymbol t}),\,\, \bar A_i({\boldsymbol s},{\boldsymbol t}) \right) \\
{\mathbb C}{\rm ov}\left( \bar A_i({\boldsymbol t}-{\boldsymbol t}_i,{\boldsymbol t}),\,\, \bar A_i({\boldsymbol s},{\boldsymbol t}) \right) & \Var \left( \bar A_i({\boldsymbol s},{\boldsymbol t}) \right)
\end{pmatrix}^{-1}\left(y,\, z\right)^\top.
\end{equation}
Combining this with   \eqref{eq:basicProbability} and \eqref{eq:CramerBasic}, we get that
\begin{align}\label{eq:infInf}
 \inf\limits_{f\in\mathcal{U}_{{\boldsymbol t},{\boldsymbol s}}} \big\{ \mathbb{I}(f) \big\} &= \inf\Big\{ \Lambda_{{\boldsymbol t},{\boldsymbol s}}(y,z) : y\geq b - \left(\mu_i-\overline\lambda_i\right) {\boldsymbol t}_i;\,\, z\geq b - (\mu_i-\overline\lambda_i) {\boldsymbol t}_i - c_i({\boldsymbol t},{\boldsymbol s}) \Big\}.
\end{align}
Since $\Lambda_{{\boldsymbol t},{\boldsymbol s}}$ is quadratic and the constraints are linear, it follows by standard calculus  that the optimal values of $y$ and $z$ are
\begin{align}\label{eq:yStar}
 y^* &\triangleq \max\left\{ b - (\mu_i-\overline\lambda_i) {\boldsymbol t}_i,\,\, \left[\frac{{\mathbb C}{\rm ov}\left( \bar A_i({\boldsymbol t}-{\boldsymbol t}_i,{\boldsymbol t}),\,\, \bar A_i({\boldsymbol s},{\boldsymbol t}) \right)}{\Var \left( \bar A_i({\boldsymbol s},{\boldsymbol t}) \right)} \right]z^* \right\}.
\end{align}
and
\begin{align}\label{eq:zStar}
 z^* &\triangleq \max\left\{ b - (\mu_i-\overline\lambda_i) {\boldsymbol t}_i - c_i({\boldsymbol t},{\boldsymbol s}),\,\, \left[\frac{{\mathbb C}{\rm ov}\left( \bar A_i({\boldsymbol t}-{\boldsymbol t}_i,{\boldsymbol t}),\,\, \bar A_i({\boldsymbol s},{\boldsymbol t}) \right)}{\Var \left( \bar A_i({\boldsymbol t}-{\boldsymbol t}_i,{\boldsymbol t}) \right)}\right] y^* \right\},
\end{align}
respectively. Although this gives four possible combinations for $(y^*,z^*)$, the following lemma states that one of them is not possible.

\begin{claim}\label{cla:oneImpossible}
For all ${\boldsymbol t}\in\mathcal{T}_i$ and ${\boldsymbol s}\in\mathcal{S}_i({\boldsymbol t})$ such that ${\boldsymbol s}\neq{\boldsymbol t}-{\boldsymbol t}_i$, we have that
\[ y^* = b - \left(\mu_i-\overline\lambda_i\right) {\boldsymbol t}_i, \qquad \text{and/or} \qquad  z^*= b - (\mu_i-\overline\lambda_i) {\boldsymbol t}_i - c_i({\boldsymbol t},{\boldsymbol s}). \]
\end{claim}
\begin{proof}
Suppose that
\begin{equation}\label{eq:yStar}
 y^* = \left[\frac{{\mathbb C}{\rm ov}\left( \bar A_i({\boldsymbol t}-{\boldsymbol t}_i,{\boldsymbol t}),\,\, \bar A_i({\boldsymbol s},{\boldsymbol t}) \right)}{\Var \left( \bar A_i({\boldsymbol s},{\boldsymbol t}) \right)} \right]z^* > b - \left(\mu_i-\overline\lambda_i\right) {\boldsymbol t}_i,
\end{equation}
and that
\begin{align*}
 z^* &= \left[\frac{{\mathbb C}{\rm ov}\left( \bar A_i({\boldsymbol t}-{\boldsymbol t}_i,{\boldsymbol t}),\,\, \bar A_i({\boldsymbol s},{\boldsymbol t}) \right)}{\Var \left( \bar A_i({\boldsymbol t}-{\boldsymbol t}_i,{\boldsymbol t}) \right)}\right] y^* >  b - (\mu_i-\overline\lambda_i) {\boldsymbol t}_i - c_i({\boldsymbol t},{\boldsymbol s}).
\end{align*}
Then, we have
\[ y^* = \left[\frac{{\mathbb C}{\rm ov}\left( \bar A_i({\boldsymbol t}-{\boldsymbol t}_i,{\boldsymbol t}),\,\, \bar A_i({\boldsymbol s},{\boldsymbol t}) \right)^2}{\Var \left( \bar A_i({\boldsymbol t}-{\boldsymbol t}_i,{\boldsymbol t}) \right) \Var \left( \bar A_i({\boldsymbol s},{\boldsymbol t}) \right)}\right] y^*, \]
which is impossible because the Cauchy-Schwarz inequality implies that
\begin{equation}\label{eq:CS}
 \frac{{\mathbb C}{\rm ov}\left( \bar A_i({\boldsymbol t}-{\boldsymbol t}_i,{\boldsymbol t}),\,\, \bar A_i({\boldsymbol s},{\boldsymbol t}) \right)^2}{\Var \left( \bar A_i({\boldsymbol t}-{\boldsymbol t}_i,{\boldsymbol t}) \right) \Var \left( \bar A_i({\boldsymbol s},{\boldsymbol t}) \right)} <1,
\end{equation}
for all ${\boldsymbol t}$, and ${\boldsymbol s}$ such that ${\boldsymbol s}\neq {\boldsymbol t}-{\boldsymbol t}_i$.
\end{proof}

Combining Claim \ref{cla:oneImpossible} with   \eqref{eq:yStar} and \eqref{eq:zStar}, we conclude that
$z^* > b - (\mu_i-\overline\lambda_i) {\boldsymbol t}_i - c_i({\boldsymbol t},{\boldsymbol s}) $
if and only if
\begin{align*}
  &b - (\mu_i-\overline\lambda_i) {\boldsymbol t}_i - c_i({\boldsymbol t},{\boldsymbol s}) < \left[\frac{{\mathbb C}{\rm ov}\left( \bar A_i({\boldsymbol t}-{\boldsymbol t}_i,{\boldsymbol t}),\,\, \bar A_i({\boldsymbol s},{\boldsymbol t}) \right)}{\Var \left( \bar A_i({\boldsymbol t}-{\boldsymbol t}_i,{\boldsymbol t}) \right)} \right] \Big[ b - \left(\mu_i-\overline\lambda_i\right) {\boldsymbol t}_i \Big],
\end{align*}
which is equivalent to
\begin{align*}
  & c_i({\boldsymbol t},{\boldsymbol s}) > \left[\frac{{\mathbb C}{\rm ov}\left( \bar A_i({\boldsymbol t}-{\boldsymbol t}_i,{\boldsymbol t}),\,\, \bar A_i({\boldsymbol t}-{\boldsymbol t}_i,{\boldsymbol s}) \right)}{\Var \left( \bar A_i({\boldsymbol t}-{\boldsymbol t}_i,{\boldsymbol t}) \right)}\right] \Big[ b - \left(\mu_i-\overline\lambda_i\right) {\boldsymbol t}_i \Big] = k_b^i({\boldsymbol t},{\boldsymbol s}).
\end{align*}
In that case, substituting the optimal values
\begin{align*}y^*&= b - (\mu_i-\overline\lambda_i) {\boldsymbol t}_i ,\\
z^*&= \left[\frac{{\mathbb C}{\rm ov}\left( \bar A_i({\boldsymbol t}-{\boldsymbol t}_i,{\boldsymbol t}),\,\, \bar A_i({\boldsymbol s},{\boldsymbol t}) \right)}{\Var \left( \bar A_i({\boldsymbol t}-{\boldsymbol t}_i,{\boldsymbol t}) \right)} \right] \Big[ b - \left(\mu_i-\overline\lambda_i\right) {\boldsymbol t}_i \Big] \end{align*}
in \eqref{eq:Lambda}, we obtain
\begin{align*}
 \Lambda_{{\boldsymbol t},{\boldsymbol s}}(y^*,z^*) &= \frac{{y^*}^2 \Var \left( \bar A_i({\boldsymbol s},{\boldsymbol t}) \right) - 2 y^*z^* {\mathbb C}{\rm ov}\left( \bar A_i({\boldsymbol t}-{\boldsymbol t}_i,{\boldsymbol t}),\,\, \bar A_i({\boldsymbol s},{\boldsymbol t}) \right) + {z^*}^2 \Var \left( \bar A_i({\boldsymbol t}-{\boldsymbol t}_i,{\boldsymbol t}) \right) }{2\left[ \Var \left( \bar A_i({\boldsymbol t}-{\boldsymbol t}_i,{\boldsymbol t}) \right) \Var \left( \bar A_i({\boldsymbol s},{\boldsymbol t}) \right) - {\mathbb C}{\rm ov}\left( \bar A_i({\boldsymbol t}-{\boldsymbol t}_i,{\boldsymbol t}),\,\, \bar A_i({\boldsymbol s},{\boldsymbol t}) \right)^2 \right]} \\
 &= \frac{ \left[\Var \left( \bar A_i({\boldsymbol s},{\boldsymbol t}) \right) - 2 \frac{{\mathbb C}{\rm ov}\left( \bar A_i({\boldsymbol t}-{\boldsymbol t}_i,{\boldsymbol t}),\,\, \bar A_i({\boldsymbol s},{\boldsymbol t}) \right)^2}{\Var \left( \bar A_i({\boldsymbol t}-{\boldsymbol t}_i,{\boldsymbol t}) \right)} + \frac{{\mathbb C}{\rm ov}\left( \bar A_i({\boldsymbol t}-{\boldsymbol t}_i,{\boldsymbol t}),\,\, \bar A_i({\boldsymbol s},{\boldsymbol t}) \right)^2}{\Var \left( \bar A_i({\boldsymbol t}-{\boldsymbol t}_i,{\boldsymbol t}) \right)}\right]\Big[ b - \left(\mu_i -\overline\lambda_i\right){\boldsymbol t}_i \Big]^2}{2\left[ \Var \left( \bar A_i({\boldsymbol t}-{\boldsymbol t}_i,{\boldsymbol t}) \right) \Var \left( \bar A_i({\boldsymbol s},{\boldsymbol t}) \right) - {\mathbb C}{\rm ov}\left( \bar A_i({\boldsymbol t}-{\boldsymbol t}_i,{\boldsymbol t}),\,\, \bar A_i({\boldsymbol s},{\boldsymbol t}) \right)^2 \right]}  \\
 &= \frac{\Big[ b - (\mu_i-\overline\lambda_i) {\boldsymbol t}_i \Big]^2}{2\,\Var \left( \bar A_i({\boldsymbol t}-{\boldsymbol t}_i,{\boldsymbol t}) \right)}.
\end{align*}
Combining this with  \eqref{eq:infInf} we get that, if
\begin{equation}\label{eq:kCondition}
 k_b^i({\boldsymbol t},{\boldsymbol s}) < c_i({\boldsymbol t},{\boldsymbol s}),
\end{equation}
then
\[ \inf\limits_{f\in\mathcal{U}_{{\boldsymbol t},{\boldsymbol s}}} \big\{ \mathbb{I}(f) \big\} = \frac{\Big[ b - (\mu_i-\overline\lambda_i) {\boldsymbol t}_i \Big]^2}{2\,\Var \left( \bar A_i({\boldsymbol t}-{\boldsymbol t}_i,{\boldsymbol t}) \right)}. \]

On the other hand, combining Claim \ref{cla:oneImpossible} with equations \eqref{eq:yStar} and \eqref{eq:zStar}, we also get that
\[ y^* > b - (\mu_i-\overline\lambda_i) {\boldsymbol t}_i \]
if and only if
\begin{align*}
  & \left[\frac{{\mathbb C}{\rm ov}\left( \bar A_i({\boldsymbol t}-{\boldsymbol t}_i,{\boldsymbol t}),\,\, \bar A_i({\boldsymbol s},{\boldsymbol t}) \right)}{\Var \left( \bar A_i({\boldsymbol s},{\boldsymbol t}) \right)} \right] \left[ b - (\mu_i-\overline\lambda_i) {\boldsymbol t}_i - c_i({\boldsymbol t},{\boldsymbol s}) \right] > b - \left(\mu_i-\overline\lambda_i\right) {\boldsymbol t}_i,
\end{align*}
which is equivalent to
\begin{align*}
  & c_i({\boldsymbol t},{\boldsymbol s}) < \left[\frac{{\mathbb C}{\rm ov}\left( \bar A_i({\boldsymbol s},{\boldsymbol t}),\,\, \bar A_i({\boldsymbol t}-{\boldsymbol t}_i,{\boldsymbol s}) \right)}{\Var \left( \bar A_i({\boldsymbol s},{\boldsymbol t}) \right)}\right] \Big[ b - \left(\mu_i-\overline\lambda_i\right) {\boldsymbol t}_i - c_i({\boldsymbol t},{\boldsymbol s}) \Big] = h_b^i({\boldsymbol t},{\boldsymbol s}).
\end{align*}
In that case, substituting the optimal values
\begin{align*} z^*&= b - \left(\mu_i-\overline\lambda_i\right) {\boldsymbol t}_i - c_i({\boldsymbol t},{\boldsymbol s}) ,\\
 y^*&= \left[\frac{{\mathbb C}{\rm ov}\left( \bar A_i({\boldsymbol t}-{\boldsymbol t}_i,{\boldsymbol t}),\,\, \bar A_i({\boldsymbol s},{\boldsymbol t}) \right)}{\Var \left( \bar A_i({\boldsymbol s},{\boldsymbol t}) \right)} \right] \left[ b - (\mu_i-\overline\lambda_i) {\boldsymbol t}_i - c_i({\boldsymbol t},{\boldsymbol s}) \right] \end{align*}
in \eqref{eq:Lambda}, we obtain that $\Lambda_{{\boldsymbol t},{\boldsymbol s}}(y^*,z^*)$ equals
\begin{align*}
&\frac{{y^*}^2 \Var \left( \bar A_i({\boldsymbol s},{\boldsymbol t}) \right) - 2 y^*z^* {\mathbb C}{\rm ov}\left( \bar A_i({\boldsymbol t}-{\boldsymbol t}_i,{\boldsymbol t}),\,\, \bar A_i({\boldsymbol s},{\boldsymbol t}) \right) + {z^*}^2 \Var \left( \bar A_i({\boldsymbol t}-{\boldsymbol t}_i,{\boldsymbol t}) \right) }{2\left[ \Var \left( \bar A_i({\boldsymbol t}-{\boldsymbol t}_i,{\boldsymbol t}) \right) \Var \left( \bar A_i({\boldsymbol s},{\boldsymbol t}) \right) - {\mathbb C}{\rm ov}\left( \bar A_i({\boldsymbol t}-{\boldsymbol t}_i,{\boldsymbol t}),\,\, \bar A_i({\boldsymbol s},{\boldsymbol t}) \right)^2 \right]} \\
 &= \frac{ \left[ \frac{{\mathbb C}{\rm ov}\left( \bar A_i({\boldsymbol t}-{\boldsymbol t}_i,{\boldsymbol t}),\,\, \bar A_i({\boldsymbol s},{\boldsymbol t}) \right)^2}{\Var \left( \bar A_i({\boldsymbol s},{\boldsymbol t}) \right)} - 2 \frac{{\mathbb C}{\rm ov}\left( \bar A_i({\boldsymbol t}-{\boldsymbol t}_i,{\boldsymbol t}),\,\, \bar A_i({\boldsymbol s},{\boldsymbol t}) \right)^2}{\Var \left( \bar A_i({\boldsymbol s},{\boldsymbol t}) \right)} + \Var \left( \bar A_i({\boldsymbol t}-{\boldsymbol t}_i,{\boldsymbol t}) \right)\right]\Big[ b - \left(\mu_i -\overline\lambda_i\right){\boldsymbol t}_i - c_i({\boldsymbol t},{\boldsymbol s}) \Big]^2}{2\left[ \Var \left( \bar A_i({\boldsymbol t}-{\boldsymbol t}_i,{\boldsymbol t}) \right) \Var \left( \bar A_i({\boldsymbol s},{\boldsymbol t}) \right) - {\mathbb C}{\rm ov}\left( \bar A_i({\boldsymbol t}-{\boldsymbol t}_i,{\boldsymbol t}),\,\, \bar A_i({\boldsymbol s},{\boldsymbol t}) \right)^2 \right]}  \\
 &=\frac{\left[ b - (\mu_i-\overline\lambda_i) {\boldsymbol t}_i - c_i({\boldsymbol t},{\boldsymbol s}) \right]^2}{2\,\Var \left( \bar A_i({\boldsymbol s},{\boldsymbol t}) \right)}.
\end{align*}
Combining this with   \eqref{eq:infInf} we get that, if
\begin{equation}\label{eq:hCondition}
 h_b^i({\boldsymbol t},{\boldsymbol s}) > c_i({\boldsymbol t},{\boldsymbol s}),
\end{equation}
then
\[ \inf\limits_{f\in\mathcal{U}_{{\boldsymbol t},{\boldsymbol s}}} \big\{ \mathbb{I}(f) \big\} = \frac{\left[ b - (\mu_i-\overline\lambda_i) {\boldsymbol t}_i - c_i({\boldsymbol t},{\boldsymbol s}) \right]^2}{2\,\Var \left( \bar A_i({\boldsymbol s},{\boldsymbol t}) \right)}. \]

Finally, if neither   \eqref{eq:kCondition} nor   \eqref{eq:hCondition} hold, Claim \ref{cla:oneImpossible} implies that
\begin{align*} y^*&= b - (\mu_i-\overline\lambda_i) {\boldsymbol t}_i,\\
 z^* &= b - (\mu_i-\overline\lambda_i) {\boldsymbol t}_i - c_i({\boldsymbol t},{\boldsymbol s}). \end{align*}
Combining this with   \eqref{eq:infInf}, we obtain that $\Lambda_{{\boldsymbol t},{\boldsymbol s}}(y^*,z^*) $ equals
\begin{align*}
 & \frac{{y^*}^2 \Var \left( \bar A_i({\boldsymbol s},{\boldsymbol t}) \right) - 2 y^*z^* {\mathbb C}{\rm ov}\left( \bar A_i({\boldsymbol t}-{\boldsymbol t}_i,{\boldsymbol t}),\,\, \bar A_i({\boldsymbol s},{\boldsymbol t}) \right) + {z^*}^2 \Var \left( \bar A_i({\boldsymbol t}-{\boldsymbol t}_i,{\boldsymbol t}) \right) }{2\left[ \Var \left( \bar A_i({\boldsymbol t}-{\boldsymbol t}_i,{\boldsymbol t}) \right) \Var \left( \bar A_i({\boldsymbol s},{\boldsymbol t}) \right) - {\mathbb C}{\rm ov}\left( \bar A_i({\boldsymbol t}-{\boldsymbol t}_i,{\boldsymbol t}),\,\, \bar A_i({\boldsymbol s},{\boldsymbol t}) \right)^2 \right]} \\
 &= \frac{{y^*}^2 \Var \left( \bar A_i({\boldsymbol s},{\boldsymbol t}) \right) \Var \left( \bar A_i({\boldsymbol t}-{\boldsymbol t}_i,{\boldsymbol t}) \right) - 2 y^*z^* {\mathbb C}{\rm ov}\left( \bar A_i({\boldsymbol t}-{\boldsymbol t}_i,{\boldsymbol t}),\,\, \bar A_i({\boldsymbol s},{\boldsymbol t}) \right) \Var \left( \bar A_i({\boldsymbol t}-{\boldsymbol t}_i,{\boldsymbol t}) \right) + {z^*}^2 \Var \left( \bar A_i({\boldsymbol t}-{\boldsymbol t}_i,{\boldsymbol t}) \right)^2 }{2\left[ \Var \left( \bar A_i({\boldsymbol t}-{\boldsymbol t}_i,{\boldsymbol t}) \right) \Var \left( \bar A_i({\boldsymbol s},{\boldsymbol t}) \right) - {\mathbb C}{\rm ov}\left( \bar A_i({\boldsymbol t}-{\boldsymbol t}_i,{\boldsymbol t}),\,\, \bar A_i({\boldsymbol s},{\boldsymbol t}) \right)^2 \right] \Var \left( \bar A_i({\boldsymbol t}-{\boldsymbol t}_i,{\boldsymbol t}) \right)} \\
 &= \frac{{y^*}^2}{2\,\Var \Big( \bar A_i({\boldsymbol t}-{\boldsymbol t}_i,{\boldsymbol t}) \Big)} + \frac{\left[ z^* \Var \left( \hat A_i({\boldsymbol t}-{\boldsymbol t}_i,{\boldsymbol t}) \right) - y^* {\mathbb C}{\rm ov}\left( \bar A_i({\boldsymbol t}-{\boldsymbol t}_i,{\boldsymbol t}),\,\, \bar A_i({\boldsymbol s},{\boldsymbol t}) \right) \right]^2 }{2\left[ \Var \left( \bar A_i({\boldsymbol t}-{\boldsymbol t}_i,{\boldsymbol t}) \right) \Var \left( \bar A_i({\boldsymbol s},{\boldsymbol t}) \right) - {\mathbb C}{\rm ov}\left( \bar A_i({\boldsymbol t}-{\boldsymbol t}_i,{\boldsymbol t}),\,\, \bar A_i({\boldsymbol s},{\boldsymbol t}) \right)^2 \right] \Var \left( \hat A_i({\boldsymbol t}-{\boldsymbol t}_i,{\boldsymbol t}) \right)} \\
 & = \frac{\Big[ b - \left(\mu_i-\overline\lambda_i\right) {\boldsymbol t}_i \Big]^2}{2\,\Var \Big( \bar A_i({\boldsymbol t}-{\boldsymbol t}_i,{\boldsymbol t}) \Big)} + \frac{\left[ \big[ b - \left(\mu_i-\overline\lambda_i\right) {\boldsymbol t}_i \big] {\mathbb C}{\rm ov}\left( \bar A_i({\boldsymbol t}-{\boldsymbol t}_i,{\boldsymbol t}),\,\, \bar A_i({\boldsymbol t}-{\boldsymbol t}_i,{\boldsymbol s}) \right) - c_i({\boldsymbol t},{\boldsymbol s}) \Var \left( \hat A_i({\boldsymbol t}-{\boldsymbol t}_i,{\boldsymbol t}) \right) \right]^2 }{2\left[ \Var \left( \bar A_i({\boldsymbol t}-{\boldsymbol t}_i,{\boldsymbol t}) \right) \Var \left( \bar A_i({\boldsymbol s},{\boldsymbol t}) \right) - {\mathbb C}{\rm ov}\left( \bar A_i({\boldsymbol t}-{\boldsymbol t}_i,{\boldsymbol t}),\,\, \bar A_i({\boldsymbol s},{\boldsymbol t}) \right)^2 \right] \Var \left( \hat A_i({\boldsymbol t}-{\boldsymbol t}_i,{\boldsymbol t}) \right)} \\
 & = \frac{\Big[ b - \left(\mu_i-\overline\lambda_i\right) {\boldsymbol t}_i \Big]^2}{2\,\Var \Big( \bar A_i({\boldsymbol t}-{\boldsymbol t}_i,{\boldsymbol t}) \Big)} + \frac{\left[ \frac{{\mathbb C}{\rm ov}\left( \bar A_i({\boldsymbol t}-{\boldsymbol t}_i,{\boldsymbol t}),\,\, \bar A_i({\boldsymbol t}-{\boldsymbol t}_i,{\boldsymbol s}) \right)}{\Var \left( \hat A_i({\boldsymbol t}-{\boldsymbol t}_i,{\boldsymbol t}) \right)} [b - (\mu_i-\overline\lambda_i) {\boldsymbol t}_i]  - c_i({\boldsymbol t},{\boldsymbol s}) \right]^2 }{2\left[ 1 - \frac{{\mathbb C}{\rm ov}\left( \bar A_i({\boldsymbol t}-{\boldsymbol t}_i,{\boldsymbol t}),\,\, \bar A_i({\boldsymbol s},{\boldsymbol t}) \right)^2}{\Var \left( \bar A_i({\boldsymbol t}-{\boldsymbol t}_i,{\boldsymbol t}) \right) \Var \left( \bar A_i({\boldsymbol s},{\boldsymbol t}) \right)} \right]  \Var \left( \bar A_i({\boldsymbol s},{\boldsymbol t}) \right) } \\
 & =\frac{\Big[ b - \left(\mu_i-\overline\lambda_i\right) {\boldsymbol t}_i \Big]^2}{2\,\Var \Big( \bar A_i({\boldsymbol t}-{\boldsymbol t}_i,{\boldsymbol t}) \Big)} + \frac{\Big[ k_b^i({\boldsymbol t},{\boldsymbol s})- c_i({\boldsymbol t},{\boldsymbol s}) \Big]^2}{2\,\Var \Big( \bar A_i({\boldsymbol s},{\boldsymbol t}) \,\Big|\, \bar A_i({\boldsymbol t}-{\boldsymbol t}_i,{\boldsymbol t}) = b - \left(\mu_i-\overline\lambda_i\right) {\boldsymbol t}_i \Big)}.
\end{align*}
Combining this with  \eqref{eq:infInf} we get that, if
\begin{equation*}
  k_b^i({\boldsymbol t},{\boldsymbol s}) \geq c_i({\boldsymbol t},{\boldsymbol s}) \qquad \text{and} \qquad h_b^i({\boldsymbol t},{\boldsymbol s}) \leq c_i({\boldsymbol t},{\boldsymbol s}),
\end{equation*}
then
\begin{align*}
 \inf\limits_{f\in\mathcal{U}_{{\boldsymbol t},{\boldsymbol s}}} \big\{ \mathbb{I}(f) \big\} & = \frac{\Big[ b - \left(\mu_i-\overline\lambda_i\right) {\boldsymbol t}_i \Big]^2}{2\,\Var \Big( \bar A_i({\boldsymbol t}-{\boldsymbol t}_i,{\boldsymbol t}) \Big)} + \frac{\Big[ k_b^i({\boldsymbol t},{\boldsymbol s})- c_i({\boldsymbol t},{\boldsymbol s}) \Big]^2}{2\,\Var \Big( \bar A_i({\boldsymbol s},{\boldsymbol t}) \,\Big|\, \bar A_i({\boldsymbol t}-{\boldsymbol t}_i,{\boldsymbol t}) = b - \left(\mu_i-\overline\lambda_i\right) {\boldsymbol t}_i \Big)},
\end{align*}
as desired
\end{proof}

Combining Lemmas \ref{lem:miniLowerBound} and \ref{lem:computeI} concludes the proof of Theorem \ref{thm:acyclicUpperLD}.

\section{Proof of Theorem \ref{thm:acyclicLowerLD1}}\label{app:acyclicLowerLD1}

Given Theorem \ref{thm:acyclicUpperLD}, it is enough to show that, if
\begin{equation}\label{eq:auxK}
 k_b^i\left({\boldsymbol t}^*,{\boldsymbol s}\right) < c_i\left({\boldsymbol t}^*,{\boldsymbol s}\right),
\end{equation}
for all ${\boldsymbol s}\in\mathcal{S}_i({\boldsymbol t}^*)$ such that ${\boldsymbol s}\neq{\boldsymbol t}^*-{\boldsymbol t}^*_i$, then
\begin{equation*}
   -\lim\limits_{n\to\infty} \frac{1}{n}\log  \mathbb{P}\left( Q_i^{(n)}  > b n \right)  \leq \inf\limits_{{\boldsymbol t}\in\mathcal{T}_i} \left\{ \frac{\Big[ b - \left(\mu_i-\overline\lambda_i\right) {\boldsymbol t}_i \Big]^2}{2\,\Var \left( \bar A_i({\boldsymbol t}-{\boldsymbol t}_i,{\boldsymbol t}) \right)} \right\}.
\end{equation*}
In the proof of Theorem \ref{thm:acyclicUpperLD}, the lower bound in the decay rate was obtained by replacing the decay rate of an intersection of events by the decay rate of the least likely of these. Therefore, if the optimum path in this least likely set happens to be in all the sets in the intersection, then the bound is tight. In particular, if ${\boldsymbol t}^*$ and ${\boldsymbol s}^*$ are optimizers in the lower bound of Theorem \ref{thm:acyclicUpperLD}, then we need to show that the most probable path in $\mathcal{U}_{{\boldsymbol t}^*,{\boldsymbol s}^*}$ is in $\mathcal{E}^i(b)$. Furthermore, since Theorem \ref{thm:Icontinuity} states that $\mathcal{E}^i(b)$ is a $\mathbb{I}$-continuity set, then it is enough to show that the most probable path in $\mathcal{U}_{{\boldsymbol t}^*,{\boldsymbol s}^*}$ is in~$\overline{\mathcal{E}^i(b)}$.

\begin{claim}
If $k_b^i\left({\boldsymbol t}^*,{\boldsymbol s}\right) < c_i\left({\boldsymbol t}^*,{\boldsymbol s}\right)$, for all ${\boldsymbol s}\in\mathcal{S}_i({\boldsymbol t}^*)$ such that ${\boldsymbol s}\neq{\boldsymbol t}^*-{\boldsymbol t}^*_i$, then a most probable path in $\mathcal{U}_{{\boldsymbol t}^*,{\boldsymbol s}^*}$ is $f^*\in\Omega^k$ such that
\[ f^*_j(\cdot) = \mathbb{E}\left[ \hat A_j(\cdot) \,\left|\, \bar A_i({\boldsymbol t}^*-{\boldsymbol t}^*_i,{\boldsymbol t}^*) = b - \left(\mu_i - \overline\lambda_i\right) {\boldsymbol t}_i^* \right.\right], \]
for $j\in\{1,\dots,k\}$.
\end{claim}
\begin{proof}
For $j\in\{1,\dots,k\}$, we have
\begin{align*}
 f^*_j(\cdot) &= \frac{{\mathbb C}{\rm ov}\left(\hat A_j(\cdot) ,\, \bar A_i({\boldsymbol t}^*-{\boldsymbol t}^*_i,{\boldsymbol t}^*) \right)}{\Var \left( \bar A_i({\boldsymbol t}^*-{\boldsymbol t}^*_i,{\boldsymbol t}^*) \right)} \Big[ b - \left(\mu_i - \overline\lambda_i\right) {\boldsymbol t}_i^* \Big].
\end{align*}
Then, we can write
\[ f^*(\cdot) = \left( \sum\limits_{r\in\mathcal{P}_1(i)} \Big[ K({\boldsymbol t}^*_{r},\cdot) - K({\boldsymbol t}^*_{r} - {\boldsymbol t}_i^*,\cdot) \Big].e_{r_1}\Pi_r  \right) \left[\frac{b - \left(\mu_i - \overline\lambda_i\right) {\boldsymbol t}_i^*}{\Var \left( \bar A_i({\boldsymbol t}^*-{\boldsymbol t}^*_i,{\boldsymbol t}^*) \right)} \right], \]
and thus $f^*$ is in the {\sc rkhs} $\mathcal{R}^k$. Then, we have
\begin{align*}
 \mathbb{I}(f^*) &= \frac{1}{2} \langle f^*,\, f^* \rangle_{\mathcal{R}^k}\\
 &= \frac{1}{2} \left( \sum\limits_{r\in\mathcal{P}_1(i)} \sum\limits_{r'\in\mathcal{P}_1(i)} e_{r_1}^\top. \big[ K({\boldsymbol t}^*_r,{\boldsymbol t}^*_{r'}) - K({\boldsymbol t}^*_r,{\boldsymbol t}^*_{r'}-{\boldsymbol t}^*_i) - K({\boldsymbol t}^*_r-{\boldsymbol t}^*_i, {\boldsymbol t}^*_{r'}) \right. \\
 &\qquad\qquad\qquad\qquad\qquad  + K({\boldsymbol t}^*_r-{\boldsymbol t}^*_i,{\boldsymbol t}^*_{r'}-{\boldsymbol t}^*_i) \big]. e_{r'_1} \Pi_r\Pi_{r'} \Big)  \left[\frac{b - \left(\mu_i - \overline\lambda_i\right) {\boldsymbol t}_i^*}{\Var \left( \bar A_i({\boldsymbol t}^*-{\boldsymbol t}^*_i,{\boldsymbol t}^*) \right)} \right]^2 \\
 &= \frac{1}{2} \left( \sum\limits_{r\in\mathcal{P}_1(i)} \sum\limits_{r'\in\mathcal{P}_1(i)} \left[ {\mathbb C}{\rm ov}\left( \hat A_{r_1}({\boldsymbol t}^*_r),\,\, \hat A_{r'_1}({\boldsymbol t}^*_{r'}) \right) - {\mathbb C}{\rm ov}\left( \hat A_{r_1}({\boldsymbol t}^*_r),\,\, \hat A_{r'_1}({\boldsymbol t}^*_{r'}-{\boldsymbol t}^*_i) \right) - {\mathbb C}{\rm ov}\left( \hat A_{r_1}({\boldsymbol t}^*_r-{\boldsymbol t}^*_i),\,\, \hat A_{r'_1}({\boldsymbol t}^*_{r'}) \right) \right.\right. \\
 &\qquad\qquad\qquad\qquad\qquad + \left.\left. {\mathbb C}{\rm ov}\left( \hat A_{r_1}({\boldsymbol t}^*_r-{\boldsymbol t}^*_i),\,\, \hat A_{r'_1}({\boldsymbol t}^*_{r'}-{\boldsymbol t}^*_i) \right) \right] \Pi_r\Pi_{r'} \right)  \left[\frac{b - \left(\mu_i - \overline\lambda_i\right) {\boldsymbol t}_i^*}{\Var \left( \bar A_i({\boldsymbol t}^*-{\boldsymbol t}^*_i,{\boldsymbol t}^*) \right)} \right]^2 \\
 &= \frac{1}{2} \left[ \sum\limits_{r\in\mathcal{P}_1(i)} \sum\limits_{r'\in\mathcal{P}_1(i)} {\mathbb C}{\rm ov}\left( \hat A_{r_1}({\boldsymbol t}^*_r) - \hat A_{r_1}({\boldsymbol t}^*_r-{\boldsymbol t}^*_i),\,\, \hat A_{r'_1}({\boldsymbol t}^*_{r'}) - \hat A_{r'_1}({\boldsymbol t}^*_{r'}-{\boldsymbol t}^*_i) \right) \Pi_r\Pi_{r'} \right]  \left[\frac{b - \left(\mu_i - \overline\lambda_i\right) {\boldsymbol t}_i^*}{\Var \left( \bar A_i({\boldsymbol t}^*-{\boldsymbol t}^*_i,{\boldsymbol t}^*) \right)} \right]^2 \\
 &= \frac{1}{2} \Var \left( \sum\limits_{r\in\mathcal{P}_1(i)} \left[ \hat A_{r_1}({\boldsymbol t}^*_r) - \hat A_{r_1}({\boldsymbol t}^*_r-{\boldsymbol t}^*_i) \right] \Pi_r \right)  \left[\frac{b - \left(\mu_i - \overline\lambda_i\right) {\boldsymbol t}_i^*}{\Var \left( \bar A_i({\boldsymbol t}^*-{\boldsymbol t}^*_i,{\boldsymbol t}^*) \right)} \right]^2\\
 &=\frac{\Big[ b - \left(\mu_i - \overline\lambda_i\right) {\boldsymbol t}_i^* \Big]^2}{2\,\Var \left( \bar A_i({\boldsymbol t}^*-{\boldsymbol t}^*_i,{\boldsymbol t}^*) \right)}.
\end{align*}
Since $k_b^i\left({\boldsymbol t}^*,{\boldsymbol s}\right) < c_i\left({\boldsymbol t}^*,{\boldsymbol s}\right)$ for all ${\boldsymbol s}\in\mathcal{S}_i({\boldsymbol t}^*)$ such that ${\boldsymbol s}\neq{\boldsymbol t}^*-{\boldsymbol t}^*_i$,  the expression above is equal to the lower bound in Theorem \ref{thm:acyclicUpperLD}. It follows that $f^*$ is a most probable path in the set $\mathcal{U}_{{\boldsymbol t}^*,{\boldsymbol s}^*}$.
\end{proof}

To complete the proof, we just need to show that $f^*\in\overline{\mathcal{E}^i(b)}$, i.e., we need to show that there exists ${\boldsymbol t}\in\overline{\mathcal{T}_i}$ such that
\[ f^*_i({\boldsymbol t}_i) + \sum\limits_{r\in\mathcal{P}_2(i)} \Big[f^*_{r_1}({\boldsymbol t}_{r}) - f^*_{r_1}({\boldsymbol s}_{r})\Big] \Pi_r \geq b - \left(\mu_i-\overline\lambda_i\right) {\boldsymbol t}_i - c_i({\boldsymbol t},{\boldsymbol s}), \]
for all ${\boldsymbol s}\in\mathcal{S}_i({\boldsymbol t})$. For ${\boldsymbol t}={\boldsymbol t}^*$, we have
\begin{align*}
  f^*_i(t^*_i) + \sum\limits_{r\in\mathcal{P}_2(i)} \Big[f^*_{r_1}({\boldsymbol t}^*_{r}) - f^*_{r_1}({\boldsymbol s}_{r})\Big] \Pi_r & = \mathbb{E}\left[\left. \bar A_i({\boldsymbol s},{\boldsymbol t}^*) \,\right|\, \bar A_i({\boldsymbol t}^*-{\boldsymbol t}^*_i,{\boldsymbol t}^*) = b - \left(\mu_i -\overline\lambda_i\right) {\boldsymbol t}_i^* \right] \\
  &= b - \left(\mu_i - \overline\lambda_i\right) t^*_i + \mathbb{E}\left[\left. \bar A_i({\boldsymbol s},{\boldsymbol t}^*-{\boldsymbol t}^*_i) \,\right|\, \bar A_i({\boldsymbol t}^*-{\boldsymbol t}^*_i,{\boldsymbol t}^*) = b - \left(\mu_i -\overline\lambda_i\right) {\boldsymbol t}_i^* \right] \\& = b - \left(\mu_i - \overline\lambda_i\right) t^*_i - k_b^i({\boldsymbol t}^*,{\boldsymbol s}).
\end{align*}
Finally, combining this with   \eqref{eq:auxK} and the fact that $k_b^i\left({\boldsymbol t}^*,{\boldsymbol t}^*-{\boldsymbol t}^*_i\right) = 0 = c_i\left({\boldsymbol t}^*,{\boldsymbol t}^*-{\boldsymbol t}^*_i\right)$, we obtain
\begin{align*}
  f^*_i(t^*_i) + \sum\limits_{r\in\mathcal{P}_2(i)} \Big[f^*_{r_1}({\boldsymbol t}^*_{r}) - f^*_{r_1}({\boldsymbol s}_{r})\Big] \Pi_r  &= b - \left(\mu_i-\overline\lambda_i\right) t^*_i - k_b^i({\boldsymbol t}^*,{\boldsymbol s})  \geq b - \left(\mu_i-\overline\lambda_i\right) t^*_i - c_i({\boldsymbol t}^*,{\boldsymbol s}),
\end{align*}
for all ${\boldsymbol s}\in\mathcal{S}_i({\boldsymbol t}^*)$, which concludes the proof.

\section{Proof of Lemma \ref{lem:sufficientCondition}}\label{app:sufficientCondition}

Since ${\boldsymbol t}-{\boldsymbol t}_i\in\mathcal{S}_i({\boldsymbol t})$ for all ${\boldsymbol t}\in\mathcal{T}_i$, we have
\begin{align*}
  \sup\limits_{{\boldsymbol s}\in\mathcal{S}_i({\boldsymbol t})} \left\{ \frac{\Big[b - \left(\mu_i-\overline\lambda_i\right) {\boldsymbol t}_i - c_i({\boldsymbol t},{\boldsymbol s}) \Big]^2}{2\,\Var \left( \bar A_i({\boldsymbol s},{\boldsymbol t}) \right)} \right\} & \geq \frac{\Big[ b - \left(\mu_i-\overline\lambda_i\right) {\boldsymbol t}_i \Big]^2}{2\,\Var \left( \bar A_i({\boldsymbol t}-{\boldsymbol t}_i,{\boldsymbol t}) \right)},
\end{align*}
for all ${\boldsymbol t}\in\mathcal{T}_i$. Therefore, we have
\begin{align*}
  \sup\limits_{{\bf }s\in\mathcal{S}_i({\boldsymbol t})} \Big\{I^i_b({\boldsymbol t},{\boldsymbol s})\Big\} &\geq \frac{\Big[ b - \left(\mu_i-\overline\lambda_i\right) {\boldsymbol t}_i \Big]^2}{2\,\Var \left( \bar A_i({\boldsymbol t}-{\boldsymbol t}_i,{\boldsymbol t}) \right)},
\end{align*}
and thus
\begin{align}
  \inf\limits_{{\boldsymbol t}\in\mathcal{T}_i}\sup\limits_{{\bf }s\in\mathcal{S}_i({\boldsymbol t})} \Big\{I^i_b({\boldsymbol t},{\boldsymbol s})\Big\} &\geq \inf\limits_{{\boldsymbol t}\in\mathcal{T}_i} \left\{ \frac{\Big[ b - \left(\mu_i-\overline\lambda_i\right) {\boldsymbol t}_i \Big]^2}{2\,\Var \left( \bar A_i({\boldsymbol t}-{\boldsymbol t}_i,{\boldsymbol t}) \right)} \right\} = \frac{\Big[ b - \left(\mu_i-\overline\lambda_i\right) \tilde {\boldsymbol t}_i \Big]^2}{2\,\Var \left( \bar A_i\left(\tilde {\boldsymbol t}- \tilde {\boldsymbol t}_i,\tilde{\boldsymbol t}\right) \right)}. \label{eq:achievedLowerBoundI}
\end{align}

On the other hand, since $k_b^i\left(\tilde {\boldsymbol t},{\boldsymbol s}\right) < c_i\left(\tilde {\boldsymbol t},{\boldsymbol s}\right)$ for all ${\boldsymbol s}\in\mathcal{S}_i(\tilde {\boldsymbol t})$ such that ${\boldsymbol s}\neq \tilde {\boldsymbol t}-\tilde {\boldsymbol t}_i$, we have
\begin{equation*}
  I^i_b(\tilde {\boldsymbol t},{\boldsymbol s}) = \frac{\Big[ b - \left(\mu_i-\overline\lambda_i\right) \tilde {\boldsymbol t}_i \Big]^2}{2\,\Var \left( \bar A_i\left(\tilde {\boldsymbol t}- \tilde {\boldsymbol t}_i,\tilde{\boldsymbol t}\right) \right)},
\end{equation*}
for all ${\boldsymbol s}\in\mathcal{S}_i(\tilde {\boldsymbol t})$. Combining this with   \eqref{eq:achievedLowerBoundI}, we get
\begin{equation}\label{eq:auxUpperBound}
  \inf\limits_{{\boldsymbol t}\in\mathcal{T}_i}\sup\limits_{{\bf }s\in\mathcal{S}_i({\boldsymbol t})} \Big\{I^i_b({\boldsymbol t},{\boldsymbol s})\Big\} = \frac{\Big[ b - \left(\mu_i-\overline\lambda_i\right) \tilde {\boldsymbol t}_i \Big]^2}{2\,\Var \left( \bar A_i\left(\tilde {\boldsymbol t}- \tilde {\boldsymbol t}_i,\tilde{\boldsymbol t}\right) \right)}.
\end{equation}
In particular, this means that we can pick $\tilde {\boldsymbol t} = {\boldsymbol t}^*$, and thus $k_b^i({\boldsymbol t}^*,{\boldsymbol s}) = k_b^i(\tilde {\boldsymbol t},{\boldsymbol s}) < c_i(\tilde {\boldsymbol t},{\boldsymbol s}) = c_i({\boldsymbol t}^*,{\boldsymbol s}),$ for all ${\boldsymbol s}\in\mathcal{S}_i({\boldsymbol t}^*)$ such that ${\boldsymbol s} \neq {\boldsymbol t}^*-{\boldsymbol t}^*_i$.

\section{Proof of Theorem \ref{thm:acyclicLowerLD3}}\label{app:acyclicLowerLD3}

Similarly to the proof of Theorem \ref{thm:acyclicLowerLD1}, if ${\boldsymbol t}^*$ and ${\boldsymbol s}^*$ are optimizers in the lower bound of Theorem \ref{thm:acyclicUpperLD}, we need to show that the most probable path in $\mathcal{U}_{{\boldsymbol t}^*,{\boldsymbol s}^*}$ is in $\overline{\mathcal{E}^i(b)}$. %We begin with a technical lemma.

%\begin{lemma}\label{lem:technical}
%  Under Assumption \ref{ass:concavity}, the following holds. If
%  \[ h_b^i\left({\boldsymbol t}^*,{\boldsymbol s}^*\right) \leq b-\left( \mu_i - \overline\lambda_i \right){\boldsymbol t}_i^* \]
%  and
%  \[ \sup\Big\{ k_b^i\left(\tilde {\boldsymbol t}_i,\tilde {\boldsymbol t},{\boldsymbol s}\right) - c_i\left(\tilde {\boldsymbol t}_i,\tilde {\boldsymbol t},{\boldsymbol s}\right) : {\boldsymbol s}\in \mathcal{S}_i\left(\tilde {\boldsymbol t}\right) \Big\} > 0, \]
%  then
%  \[ k_b^i\left({\boldsymbol t}^*,{\boldsymbol s}^*\right) > c_i\left({\boldsymbol t}^*,{\boldsymbol s}^*\right). \]
%\end{lemma}
%\begin{proof}
%{\color{red}The proof is analogous to the proof of Lemma 3.9 in \cite{MichelTandemPaper} (which needs Assumption \ref{ass:concavity}), and it is thus omitted.}
%\end{proof}

\begin{claim}
If $h_b^i\left({\boldsymbol t}^*,{\boldsymbol s}^*\right) \leq c_i\left({\boldsymbol t}^*,{\boldsymbol s}^*\right)$
and
$k_b^i\left({\boldsymbol t}^*,{\boldsymbol s}^*\right) \geq c_i\left({\boldsymbol t}^*,{\boldsymbol s}^*\right),$
then a most probable path in $\mathcal{U}_{{\boldsymbol t}^*,{\boldsymbol s}^*}$ is $f^*\in\Omega^k$ such that
\begin{align*}
 f^*_j(\cdot) &= \mathbb{E}\left[ \hat A_j(\cdot) \,\left|\, \bar A_i({\boldsymbol t}^*-{\boldsymbol t}^*_i,{\boldsymbol t}^*) = b - \left(\mu_i-\overline\lambda_i\right) {\boldsymbol t}_i^*;\,\, \bar A_i({\boldsymbol t}^*-{\boldsymbol t}^*_i,{\boldsymbol s}^*) = c_i({\boldsymbol t}^*,{\boldsymbol s}^*) \right. \right],
\end{align*}
for $j\in\{1,\dots,k\}$.
\end{claim}
\begin{proof}
Using standard properties of conditional multivariate Normal random variables, we get that
\begin{align*}
  f^*_j(\cdot) &= \theta^*_1 {\mathbb C}{\rm ov}\left( \hat A_j(\cdot),\,\, \bar A_i({\boldsymbol t}^*-{\boldsymbol t}^*_i,{\boldsymbol t}^*) \right) + \theta^*_2 {\mathbb C}{\rm ov}\left( \hat A_j(\cdot),\,\, \bar A_i({\boldsymbol t}^*-{\boldsymbol t}^*_i,{\boldsymbol s}^*) \right) ,
\end{align*}
for all $j\in\{1,\dots,k\}$, where
\begin{align*}
\theta^*&\triangleq
\begin{pmatrix}
 \Var \left( \bar A_i({\boldsymbol t}^*-{\boldsymbol t}^*_i,{\boldsymbol t}^*) \right) & {\mathbb C}{\rm ov}\left( \bar A_i({\boldsymbol t}^*-{\boldsymbol t}^*_i,{\boldsymbol t}^*),\,\, \bar A_i({\boldsymbol t}^*-{\boldsymbol t}^*_i,{\boldsymbol s}^*) \right) \\
{\mathbb C}{\rm ov}\left( \bar A_i({\boldsymbol t}^*-{\boldsymbol t}^*_i,{\boldsymbol t}^*),\,\, \bar A_i({\boldsymbol t}^*-{\boldsymbol t}^*_i,{\boldsymbol s}^*) \right) & \Var \left( \bar A_i({\boldsymbol t}^*-{\boldsymbol t}^*_i,{\boldsymbol s}^*) \right)
\end{pmatrix}^{-1}
\begin{pmatrix}
  b - \left(\mu_i-\overline\lambda_i\right) {\boldsymbol t}_i^* \\
  c_i({\boldsymbol t}^*,{\boldsymbol s}^*)
\end{pmatrix}.
\end{align*}
Then, we can write
\[ f^*(\cdot) = \theta^*_1 \left[ \sum\limits_{r\in\mathcal{P}_1(i)} \Big[ K({\boldsymbol t}^*_{r},\cdot) - K({\boldsymbol t}^*_{r} - {\boldsymbol t}_i^*,\cdot) \Big].e_{r_1}\Pi_r  \right] + \theta^*_2\left[ \sum\limits_{r\in\mathcal{P}_1(i)} \Big[ K({\boldsymbol s}^*_{r},\cdot) - K({\boldsymbol t}^*_{r} - {\boldsymbol t}_i^*,\cdot) \Big].e_{r_1}\Pi_r  \right], \]
and thus $f^*$ is in the {\sc rkhs} $\mathcal{R}^k$. After tedious but straightforward computations
% {\color{red}[SHOULD WE MAKE THEM EXPLICIT? THEY ARE REALLY TEDIOUS]},
we obtain
\begin{align*}
 \mathbb{I}(f^*) &= \frac{\Big[ b - \left(\mu_i-\overline\lambda_i\right) {\boldsymbol t}_i \Big]^2}{2\,\Var \Big( \bar A_i({\boldsymbol t}-{\boldsymbol t}_i,{\boldsymbol t}) \Big)} + \frac{\Big[ k_b^i(t,{\boldsymbol t},{\boldsymbol s})- c_i({\boldsymbol t},{\boldsymbol s}) \Big]^2}{2\,\Var \Big( \bar A_i({\boldsymbol t}-{\boldsymbol t}_i,{\boldsymbol s}) \,\Big|\, \bar A_i({\boldsymbol t}-{\boldsymbol t}_i,{\boldsymbol t}) = b - \left(\mu_i-\overline\lambda_i\right) {\boldsymbol t}_i \Big)}.
\end{align*}
Since $h_b^i\left({\boldsymbol t}^*,{\boldsymbol s}^*\right) \leq b-\left( \mu_i - \overline\lambda_i \right){\boldsymbol t}_i^*$ and $k_b^i\left({\boldsymbol t}^*,{\boldsymbol s}^*\right) \geq c_i\left({\boldsymbol t}^*,{\boldsymbol s}^*\right)$, the equation above is equal to the lower bound in Theorem \ref{thm:acyclicUpperLD}. It follows that $f^*$ is a most probable path in $\mathcal{U}_{{\boldsymbol t}^*,{\boldsymbol s}^*}$.
\end{proof}

To complete the proof, we just need to show that $f^*\in\overline{\mathcal{E}^i(b)}$, i.e., we need to show that there exists ${\boldsymbol t}\in\overline{\mathcal{T}_i}$ such that
\[ f^*_i({\boldsymbol t}_i) + \sum\limits_{r\in\mathcal{P}_2(i)} \Big[f^*_{r_1}({\boldsymbol t}_{r}) - f^*_{r_1}({\boldsymbol s}_{r})\Big] \Pi_r \geq b - \left(\mu_i-\overline\lambda_i\right) {\boldsymbol t}_i - c_i({\boldsymbol t},{\boldsymbol s}), \]
for all ${\boldsymbol s}\in\mathcal{S}_i({\boldsymbol t})$. In order to simplify notation, we denote
\begin{align*}
  \overline{\mathbb{E}}[ \,\,\cdot\,\, ] &\triangleq \mathbb{E}\left[ \,\,\cdot\,\, \left|\,\, \bar A_i({\boldsymbol t}^*-{\boldsymbol t}^*_i,{\boldsymbol t}^*) = b - \left(\mu_i-\overline\lambda_i\right) t^*_i; \,\, \bar A_i({\boldsymbol t}^*-{\boldsymbol t}^*_i,{\boldsymbol s}^*) = c_i({\boldsymbol t}^*,{\boldsymbol s}^*) \right.\right].
\end{align*}
For ${\boldsymbol t}={\boldsymbol t}^*$, we have
\begin{align*}
  f^*_i({\boldsymbol t}_i^*) + \sum\limits_{r\in\mathcal{P}_2(i)} \Big[f^*_{r_1}({\boldsymbol t}^*_{r}) - f^*_{r_1}({\boldsymbol s}_{r})\Big] \Pi_r &= \overline{\mathbb{E}}\left[ \hat A_i({\boldsymbol t}^*_i) + \sum\limits_{r\in\mathcal{P}_2(i)} \Big[ \hat A_{r_1}({\boldsymbol t}^*_{r}) - \hat A_{r_1}({\boldsymbol s}_{r}) \Big] \Pi_r  \right] \\
&= b - \left(\mu_i-\overline\lambda_i\right) {\boldsymbol t}_i^* - \overline{\mathbb{E}}\left[ \sum\limits_{r\in\mathcal{P}_2(i)} \Big[ \hat A_{r_1}({\boldsymbol s}_{r}) - \hat A_{r_1}({\boldsymbol t}^*_{r}-{\boldsymbol t}_i^*) \Big] \Pi_r \right].
\end{align*}
Combining this with  \eqref{eq:weirdCondition}, we obtain
\begin{align*}
 f^*_i({\boldsymbol t}_i^*) + \sum\limits_{r\in\mathcal{P}_2(i)} \Big[f^*_{r_1}({\boldsymbol t}^*_{r}) - f^*_{r_1}({\boldsymbol s}_{r})\Big] \Pi_r &\geq b - \left(\mu_i-\overline\lambda_i\right) {\boldsymbol t}_i^* - c_i({\boldsymbol t}^*,{\boldsymbol s}),
\end{align*}
for all ${\boldsymbol s}\in\mathcal{S}_i({\boldsymbol t}^*)$, which concludes the proof.

\section{Proof of Theorem \ref{thm:mfBmExample}}\label{app:mfBmExample}
We start with a technical lemma.

\begin{lemma}\label{lem:optimizerProperty}
There exists
\begin{equation}\label{eq:toOptimize}
 {\boldsymbol t}^* \in \underset{{\boldsymbol t}\in\overline{\mathcal{T}_i}}{\arg\min} \left\{ \frac{\Big[b-\big(\mu_i\overline\lambda_i\big){\boldsymbol t}_i\Big]^2}{\Var \left( \bar A_i({\boldsymbol t}-{\boldsymbol t}_i,{\boldsymbol t}) \right)} \right\}.
\end{equation}
such that ${\boldsymbol t}^*_r={\boldsymbol t}^*_i$, for all $r\in\mathcal{P}_2(i)$.
\end{lemma}
\begin{proof}
Note that the numerator of the function being minimized in   \eqref{eq:toOptimize} only depends on ${\boldsymbol t}_i$. As a result, we can focus on the structure of the maximizers of its denominator when we keep ${\boldsymbol t}_i$ fixed. Using that $\hat A(\cdot)$ is a time-reversible mfBm, we obtain that $\Var \left( \bar A_i({\boldsymbol t}-{\boldsymbol t}_i,{\boldsymbol t}) \right)$ equals
\begin{align*}
  & \sum\limits_{r\in\mathcal{P}_1(i)} \sum\limits_{r'\in\mathcal{P}_1(i)} \Pi_r \Pi_{r'} {\mathbb C}{\rm ov}\left( \hat A_{r_1}({\boldsymbol t}_{r}) - \hat A_{r_1}({\boldsymbol t}_{r}-{\boldsymbol t}_i),\,\, \hat A_{r'_1}({\boldsymbol t}_{r'}) - \hat A_{r'_1}({\boldsymbol t}_{r'}-{\boldsymbol t}_i)\right) \\
  &= \sum\limits_{r\in\mathcal{P}_1(i)} \sum\limits_{r'\in\mathcal{P}_1(i)} \Pi_r \Pi_{r'} \left[ {\mathbb C}{\rm ov}\left( \hat A_{r_1}({\boldsymbol t}_{r}),\,\, \hat A_{r'_1}({\boldsymbol t}_{r'})\right) - {\mathbb C}{\rm ov}\left( \hat A_{r_1}({\boldsymbol t}_{r}),\,\, \hat A_{r'_1}({\boldsymbol t}_{r'}-{\boldsymbol t}_i)\right) \right. \\
  &\qquad\qquad\qquad\qquad\qquad\qquad - \left. {\mathbb C}{\rm ov} \left( \hat A_{r_1}({\boldsymbol t}_{r}-{\boldsymbol t}_i),\,\, \hat A_{r'_1}({\boldsymbol t}_{r'}) \right) + {\mathbb C}{\rm ov} \left( \hat A_{r_1}({\boldsymbol t}_{r}-{\boldsymbol t}_i),\,\, \hat A_{r'_1}({\boldsymbol t}_{r'}-{\boldsymbol t}_i) \right)  \right] \\
  &=\sum\limits_{r\in\mathcal{P}_1(i)} \sum\limits_{r'\in\mathcal{P}_1(i)}  \frac{\sigma_{r'_1}\sigma_{r_1}\rho_{r_1,r'_1}}{2} \left[ \Big( |{\boldsymbol t}_{r}|^{2H} + |{\boldsymbol t}_{r'}|^{2H} - |{\boldsymbol t}_{r}-{\boldsymbol t}_{r'}|^{2H} \Big) - \Big( |{\boldsymbol t}_{r}|^{2H} + |{\boldsymbol t}_{r'}-{\boldsymbol t}_i|^{2H} - |{\boldsymbol t}_{r}-{\boldsymbol t}_{r'}+{\boldsymbol t}_i|^{2H} \Big) \right. \\
  &\qquad\qquad\qquad - \left. \Big( |{\boldsymbol t}_{r}-{\boldsymbol t}_i|^{2H} + |{\boldsymbol t}_{r'}|^{2H} - |{\boldsymbol t}_{r}-{\boldsymbol t}_i-{\boldsymbol t}_{r'}|^{2H} \Big) + \Big( |{\boldsymbol t}_{r}-{\boldsymbol t}_i|^{2H} + |{\boldsymbol t}_{r'}-{\boldsymbol t}_i|^{2H} - |{\boldsymbol t}_{r}-{\boldsymbol t}_{r'}|^{2H} \Big)  \right] \Pi_r \Pi_{r'} \\
  &= \sum\limits_{r\in\mathcal{P}_1(i)} \sum\limits_{r'\in\mathcal{P}_1(i)}  \frac{\sigma_{r'_1}\sigma_{r_1}\rho_{r_1,r'_1}}{2} \Big[ \Big( |{\boldsymbol t}_{r}-{\boldsymbol t}_{r'}+{\boldsymbol t}_i|^{2H} + |{\boldsymbol t}_{r}-{\boldsymbol t}_i-{\boldsymbol t}_{r'}|^{2H} - 2 |{\boldsymbol t}_{r}-{\boldsymbol t}_{r'}|^{2H} \Big) \Big] \Pi_r \Pi_{r'}.
\end{align*}
Taking the derivative with respect to ${\boldsymbol t}_{r}$, and using that ${\boldsymbol t}_{r}\leq {\boldsymbol t}_i \leq 0$ for all ${\boldsymbol t}\in\overline{\mathcal{T}_i}$, we obtain
\begin{align*}
  \frac{\partial}{\partial {\boldsymbol t}_{r}} \Var \left( \bar A_i({\boldsymbol t}-{\boldsymbol t}_i,{\boldsymbol t}) \right) &= \sum\limits_{r'\in\mathcal{P}_1(i)}  \sigma_{r'_1}\sigma_{r_1}\rho_{r_1,r'_1}H \Big[ sign({\boldsymbol t}_{r}-{\boldsymbol t}_{r'}+{\boldsymbol t}_i) |{\boldsymbol t}_{r}-{\boldsymbol t}_{r'}+{\boldsymbol t}_i|^{2H-1} \\
  &\qquad + sign({\boldsymbol t}_{r}-{\boldsymbol t}_{r'}-{\boldsymbol t}_i) |{\boldsymbol t}_{r}-{\boldsymbol t}_{r'}-{\boldsymbol t}_i|^{2H-1} - 2 {\rm sign}({\boldsymbol t}_{r}-{\boldsymbol t}_{r'}) |{\boldsymbol t}_{r}-{\boldsymbol t}_{r'}|^{2H-1} \Big] \Pi_r \Pi_{r'}.
\end{align*}
Moreover, for all ${\boldsymbol t}_{r}\leq \min\{ {\boldsymbol t}_{r'} : r'\in\mathcal{P}_1(i),\,\, r'\neq r \}$, we have
\begin{align}
  \frac{\partial}{\partial {\boldsymbol t}_{r}} \Var \left( \bar A_i({\boldsymbol t}-{\boldsymbol t}_i,{\boldsymbol t}) \right) &= \sum\limits_{r'\in\mathcal{P}_1(i),\, r'\neq r}  \sigma_{r'_1}\sigma_{r_1}\rho_{r_1,r'_1}H \Big[ - ({\boldsymbol t}_{r'}-{\boldsymbol t}_{r}-{\boldsymbol t}_i)^{2H-1} \label{eq:derivative1} \\
  &\qquad + {\rm sign}({\boldsymbol t}_{r}-{\boldsymbol t}_{r'}-{\boldsymbol t}_i) |{\boldsymbol t}_{r}-{\boldsymbol t}_{r'}-{\boldsymbol t}_i|^{2H-1}  + 2 ({\boldsymbol t}_{r'}-{\boldsymbol t}_{r})^{2H-1} \Big] \Pi_r \Pi_{r'}. \nonumber
\end{align}
If ${\boldsymbol t}_{r}-{\boldsymbol t}_{r'}-{\boldsymbol t}_i\leq 0$, we have
\begin{align}
 &- ({\boldsymbol t}_{r'}-{\boldsymbol t}_{r}-{\boldsymbol t}_i)^{2H-1} + {\rm sign}({\boldsymbol t}_{r}-{\boldsymbol t}_{r'}-{\boldsymbol t}_i) |{\boldsymbol t}_{r}-{\boldsymbol t}_{r'}-{\boldsymbol t}_i|^{2H-1}  + 2 ({\boldsymbol t}_{r'}-{\boldsymbol t}_{r})^{2H-1} \nonumber \\
 &\qquad\qquad \qquad\qquad\qquad = - ({\boldsymbol t}_{r'}-{\boldsymbol t}_{r}-{\boldsymbol t}_i)^{2H-1} - ({\boldsymbol t}_{r'}-{\boldsymbol t}_{r}+{\boldsymbol t}_i)^{2H-1}  + 2 ({\boldsymbol t}_{r'}-{\boldsymbol t}_{r})^{2H-1} \geq 0, \label{eq:derivative2}
\end{align}
where in the last inequality we used that $H\geq 1/2$. On the other hand, if ${\boldsymbol t}_{r}-{\boldsymbol t}_{r'}-{\boldsymbol t}_i> 0$, we have
\begin{align}
 &- |{\boldsymbol t}_{r}-{\boldsymbol t}_{r'}+{\boldsymbol t}_i|^{2H-1} + {\rm sign}({\boldsymbol t}_{r}-{\boldsymbol t}_{r'}-{\boldsymbol t}_i) |{\boldsymbol t}_{r}-{\boldsymbol t}_{r'}-{\boldsymbol t}_i|^{2H-1}  + 2 |{\boldsymbol t}_{r}-{\boldsymbol t}_{r'}|^{2H-1} \nonumber \\
 &\qquad\qquad \qquad\qquad\qquad = - ({\boldsymbol t}_{r'}-{\boldsymbol t}_{r}-{\boldsymbol t}_i)^{2H-1} + ({\boldsymbol t}_{r}-{\boldsymbol t}_{r'}-{\boldsymbol t}_i)^{2H-1}  + 2 ({\boldsymbol t}_{r'}-{\boldsymbol t}_{r})^{2H-1} \geq 0, \label{eq:derivative3}
\end{align}
where in the last inequality we used that $H\geq 1/2$. Combining   \eqref{eq:derivative1}, \eqref{eq:derivative2}, and \eqref{eq:derivative3} with   $\rho_{r_1,r'_1}\geq 0$, for all $r,r'\in\mathcal{P}_1(i)$, it follows that $\Var \left( \bar A_i({\boldsymbol t}-{\boldsymbol t}_i,{\boldsymbol t}) \right)$ is maximized when ${\boldsymbol t}_{r}={\boldsymbol t}_i$, for all $r\in\mathcal{P}_2(i)$.
\end{proof}

Lemma \ref{lem:optimizerProperty} implies that we can pick
\begin{equation*}
 {\boldsymbol t}^* \in \underset{{\boldsymbol t}\in\overline{\mathcal{T}_i}}{\arg\min} \left\{ \frac{\Big[b-\big(\mu_i\overline\lambda_i\big){\boldsymbol t}_i\Big]^2}{\Var \left( \bar A_i({\boldsymbol t}-{\boldsymbol t}_i,{\boldsymbol t}) \right)} \right\}
\end{equation*}
such that ${\boldsymbol t}^*_r={\boldsymbol t}^*_i$, for all $r\in\mathcal{P}_2(i)$. In that case, we have
\[ {\boldsymbol t}^*_i \in \underset{{\boldsymbol t}_i\leq 0}{\arg\min} \left\{ \frac{\Big[ b-\big(\mu_i - \overline\lambda_i\big){\boldsymbol t}_i \Big]^2}{\Var \left( \sum\limits_{r\in\mathcal{P}_1(i)} \hat A_{r_1}({\boldsymbol t}_i) \Pi_r \right)} \right\}. \]
An elementary computation yields  that
\begin{equation}\label{eq:tStar}
 {\boldsymbol t}_i^* = -\left( \frac{b}{\mu_i-\overline{\lambda}_i} \right)\left( \frac{H}{1-H} \right).
\end{equation}
Using this, the condition in Lemma~\ref{lem:sufficientCondition} is
\begin{align}
 & \frac{{\mathbb C}{\rm ov}\left( \sum\limits_{r\in\mathcal{P}_2(i)} \hat A_{r_1}({\boldsymbol s}_{r}) \Pi_r ,\,\, \hat A_i({\boldsymbol t}_i^*) + \sum\limits_{r\in\mathcal{P}_2(i)} \hat A_{r_1}({\boldsymbol t}_i^*) \Pi_r \right)}{\Var \left( \hat A_i({\boldsymbol t}_i^*) + \sum\limits_{r\in\mathcal{P}_2(i)} \hat A_{r_1}({\boldsymbol t}_i^*) \Pi_r \right)}\Big[b-\big(\mu_i-\overline{\lambda}_i\big) {\boldsymbol t}_i^*\Big] \label{eq:sufficient1} \\
 &\qquad \qquad\qquad\qquad\qquad\qquad\qquad\qquad\qquad < \sum\limits_{r\in\mathcal{P}_2(i)} \left(\mu_{r_1}-\lambda_{r_1}-\sum\limits_{j\in\mathcal{N}_{\rm in}(r_1)}  \mu_j p_{j,r_1} \right)(-{\boldsymbol s}_{r}) \Pi_r , \nonumber
\end{align}
for all ${\boldsymbol s}\in\mathcal{S}_i({\boldsymbol t}^*)$ such that ${\boldsymbol s}\neq {\boldsymbol t}^*-{\boldsymbol t}^*_i$. Then, since ${\boldsymbol t}^*-{\boldsymbol t}^*_i \notin \mathcal{S}_i({\boldsymbol t}^*)$, a sufficient condition for  \eqref{eq:sufficient1} to hold is that
\begin{align*}
 &\min\left\{  \mu_j - \lambda_j -\sum\limits_{l\in\mathcal{N}_{\rm in}(j)}  \mu_l p_{l,j} : j\neq i \right\} > \\
 &\qquad \qquad \sup\limits_{{\boldsymbol s}\in \mathcal{S}({\boldsymbol t}^*)} \left\{ \frac{{\mathbb C}{\rm ov}\left( \sum\limits_{r\in\mathcal{P}_2(i)} \hat A_{r_1}({\boldsymbol s}_{r}) \Pi_r ,\,\, \hat A_i({\boldsymbol t}_i^*) + \sum\limits_{r\in\mathcal{P}_2(i)} \hat A_{r_1}({\boldsymbol t}_i^*) \Pi_r \right)}{\Var \left( \hat A_i({\boldsymbol t}_i^*) + \sum\limits_{r\in\mathcal{P}_2(i)} \hat A_{r_1}({\boldsymbol t}_i^*) \Pi_r \right) \left( \sum\limits_{r\in\mathcal{P}_2(i)} -{\boldsymbol s}_{r} \Pi_r \right)}\Big[b-\big(\mu_i-\overline{\lambda}_i\big) {\boldsymbol t}_i^*\Big] \right\}.
\end{align*}
Substituting  \eqref{eq:tStar} in the equation above, we obtain, with $b_H\triangleq b/(1-H)$,
\begin{align*}
  &\frac{{\mathbb C}{\rm ov}\left( \sum\limits_{r\in\mathcal{P}_2(i)} \hat A_{r_1}({\boldsymbol s}_{r}) \Pi_r ,\,\, \hat A_i({\boldsymbol t}_i^*) + \sum\limits_{r\in\mathcal{P}_2(i)} \hat A_{r_1}({\boldsymbol t}_i^*) \Pi_r \right)}{\Var \left( \hat A_i({\boldsymbol t}_i^*) + \sum\limits_{r\in\mathcal{P}_2(i)} \hat A_{r_1}({\boldsymbol t}_i^*) \Pi_r \right) \left( \sum\limits_{r\in\mathcal{P}_2(i)} -{\boldsymbol s}_{r} \Pi_r \right)}\Big[b-\big(\mu_i-\overline{\lambda}_i\big) {\boldsymbol t}_i^*\Big] \\
  & =\frac{b_H\cdot \sum\limits_{r\in\mathcal{P}_2(i)} \left[ {\mathbb C}{\rm ov}\left( \hat A_{r_1}({\boldsymbol s}_{r}),\,\, \hat A_i({\boldsymbol t}_i^*) \right) + \sum\limits_{r'\in\mathcal{P}_2(i)} {\mathbb C}{\rm ov}\left( \hat A_{r_1}({\boldsymbol s}_{r}), \hat A_{r'_1}({\boldsymbol t}_i^*) \right)\Pi_{r'}\right] \Pi_r}{\left(\Var \left( \hat A_i({\boldsymbol t}_i^*) \right) + \sum\limits_{r\in\mathcal{P}_2(i)} \left[ 2 {\mathbb C}{\rm ov}\left( \hat A_i({\boldsymbol t}_i^*),\,\, \hat A_{r_1}({\boldsymbol t}_i^*) \right) + \sum\limits_{r'\in\mathcal{P}_2(i)} {\mathbb C}{\rm ov}\left( \hat A_{r_1}({\boldsymbol t}_i^*),\,\, \hat A_{r'_1}({\boldsymbol t}_i^*) \right) \Pi_{r'} \right] \Pi_r  \right) \left( \sum\limits_{r\in\mathcal{P}_2(i)} -{\boldsymbol s}_{r} \Pi_r \right)}  \\
  & =\frac{ b_H\cdot\sum\limits_{r\in\mathcal{P}_2(i)} \left( \sigma_{r_1}\sigma_i\rho_{r_1,i} + \sum\limits_{r'\in\mathcal{P}_2(i)} \sigma_{r_1}\sigma_{r'_1}\rho_{r_1,r'_1} \Pi_{r'}\right) \Pi_r \Big(|{\boldsymbol s}_{r}|^{2H}+|{\boldsymbol t}_i^*|^{2H}-|{\boldsymbol t}_i^*-{\boldsymbol s}_{r}|^{2H}\Big)}{2|{\boldsymbol t}_i^*|^{2H} \left[ \sigma_i^2 + \sum\limits_{r\in\mathcal{P}_2(i)} \left( 2 \sigma_{r_1}\sigma_i\rho_{r_1,i} + \sum\limits_{r'\in\mathcal{P}_2(i)} \sigma_{r_1}\sigma_{r'_1}\rho_{r_1,r'_1} \Pi_{r'}\right)\Pi_r  \right] \left( \sum\limits_{r\in\mathcal{P}_2(i)} -{\boldsymbol s}_{r} \Pi_r \right)}  \\
  & =\frac{ \sum\limits_{r\in\mathcal{P}_2(i)} \left( \sigma_{r_1}\sigma_i\rho_{r_1,i} + \sum\limits_{r'\in\mathcal{P}_2(i)} \sigma_{r_1}\sigma_{r'_1}\rho_{r_1,r'_1} \Pi_{r'}\right) \Pi_r \left(\left|\frac{{\boldsymbol s}_{r}}{{\boldsymbol t}_i^*}\right|^{2H}+ 1 - \left| 1- \frac{{\boldsymbol s}_{r}}{{\boldsymbol t}_i^*}\right|^{2H}\right)}{2 \left[ \sigma_i^2 + \sum\limits_{r\in\mathcal{P}_2(i)} \left( 2 \sigma_{r_1}\sigma_i\rho_{r_1,i} + \sum\limits_{r'\in\mathcal{P}_2(i)} \sigma_{r_1}\sigma_{r'_1}\rho_{r_1,r'_1} \Pi_{r'}\right)\Pi_r  \right] \left( \sum\limits_{r\in\mathcal{P}_2(i)} \frac{{\boldsymbol s}_{r}}{{\boldsymbol t}_i^*} \Pi_r \right)} \left(\frac{\mu_i-\overline{\lambda}_i}{H}\right).
\end{align*}
Then, a sufficient condition for   \eqref{eq:sufficient1} to hold is that
\begin{align*}
 &\min\left\{  \mu_j - \lambda_j -\sum\limits_{l\in\mathcal{N}_{\rm in}(j)}  \mu_l p_{l,j} : j\neq i \right\} > \\
 &\qquad\  \sup\limits_{\alpha\in (0,1)^{|\mathcal{P}_2(i)|}} \left\{ \frac{ \sum\limits_{r\in\mathcal{P}_2(i)} \left( \sigma_{r_1}\sigma_i\rho_{r_1,i} + \sum\limits_{r'\in\mathcal{P}_2(i)} \sigma_{r_1}\sigma_{r'_1}\rho_{r_1,r'_1} \Pi_{r'}\right) \Pi_r \left(\left(\alpha_r\right)^{2H}+ 1 - \left( 1- \alpha_r\right)^{2H}\right)}{ \left[ \sigma_i^2 + \sum\limits_{r\in\mathcal{P}_2(i)} \left( 2 \sigma_{r_1}\sigma_i\rho_{r_1,i} + \sum\limits_{r'\in\mathcal{P}_2(i)} \sigma_{r_1}\sigma_{r'_1}\rho_{r_1,r'_1} \Pi_{r'}\right)\Pi_r  \right] \left( \sum\limits_{r\in\mathcal{P}_2(i)} \alpha_r \Pi_r \right)} \left(\frac{\mu_i-\overline{\lambda}_i}{2H}\right) \right\}.
\end{align*}
Lemma \ref{lem:sufficientCondition} and Theorem \ref{thm:acyclicLowerLD1} finish the proof.

\bibliographystyle{unsrt}
\bibliography{references}

\end{document}